\documentclass[11pt,dvips,twoside]{article}

\usepackage{pslatex}
\usepackage{fancyhdr}
\usepackage{graphicx}
\usepackage{geometry}

\RequirePackage{amsfonts,amssymb,amsmath,amscd,amsthm}
\RequirePackage{txfonts}
\RequirePackage{graphicx}
\RequirePackage{xcolor}
\RequirePackage{geometry}
\RequirePackage{enumerate}

\def\figurename{Figure} % Replace the colon that normally appears after the Figure number by a period.
\makeatletter
\renewcommand{\fnum@figure}[1]{\figurename~\thefigure.}
\makeatother

\def\tablename{Table} % Replace the colon that normally appears after the Figure number by a period.
\makeatletter
\renewcommand{\fnum@table}[1]{\tablename~\thetable.}
\makeatother

\usepackage{amsmath}
\usepackage{amssymb}
\usepackage{amsfonts}
\usepackage{amsthm,amscd}

\newtheorem{theorem}{Theorem}[section]
\newtheorem{lemma}[theorem]{Lemma}
\newtheorem{corollary}[theorem]{Corollary}
\newtheorem{proposition}[theorem]{Proposition}

\theoremstyle{example}
\newtheorem{example}[theorem]{Example}

\theoremstyle{definition}
\newtheorem{definition}[theorem]{Definition}

\theoremstyle{remark}
\newtheorem{remark}[theorem]{Remark}

\numberwithin{equation}{section}

\begin{document}

\title{\bfseries\scshape{On BiHom-associative dialgebras}}

\author{\bfseries\scshape Ahmed Zahari\thanks{e-mail address: zaharymaths@gmail.com}\\
Universit\'{e} de Haute Alsace,\\
 IRIMAS-D\'{e}partement de Math\'{e}matiques,\\
 6, rue des Fr\`eres Lumi\`ere F-68093 Mulhouse, France.\\
\bfseries\scshape Ibrahima BAKAYOKO\thanks{e-mail address: ibrahimabakayoko27@gmail.com}\\
D\'epartement de Math\'ematiques,
Universit\'e de N'Z\'er\'ekor\'e,\\
BP 50 N'Z\'er\'ekor\'e, Guin\'ee.
}

\date{}
\maketitle 

% % % % % % \thispagestyle{empty} \setcounter{page}{1}
% % % % % % % ------- [First Page Running Head] - place it immediately after title! ------
% % % % % % \thispagestyle{fancy} \fancyhead{}
% % % % % % \fancyhead[L]{{\LARGE A}frican {\LARGE D}iaspora {\LARGE J}ournal of {\LARGE M}athematics\\
% % % % % % Volume X, Number X, pp. {\thepage--\pageref{lastpage-01} (2013)}} % put \label{lastpage-xx} on the last page!
% % % % % % \fancyhead[R]{ISSN 1539-854X  \\ {\tt{www.math-res-pub.org/adjm}}}
% % % % % % \fancyfoot{}
% % % % % % \renewcommand{\headrulewidth}{0pt}
%------------------------------------------------------------------------------

\noindent\hrulefill

\noindent {\bf Abstract.} 
The aim of this paper is to introduce and study BiHom-associative dialgebras. We give various constructions and study its connections with
 BiHom-Poisson dialgebras and BiHom-Leibniz algebras. 
% In this paper, we introduce both  BiHom-diassociative and BiHom-Poisson diassociative algebras. We discuss these new structures by
%  presenting some basic properties and constructions. We build one important class of BiHom-diassociative algebras and give properties 
% of right and left operations in BiHom-diassociative algebras. Then we prove that any BiHom-diassociative algebra,  together with a Rota-Baxter 
% operator, gives rise to another BiHom-diassociative.
 Next we discuss the central extensions of BiHom-diassociative and we describe the  
classification of $n$-dimensional BiHom-diassociative algebras for 
$n\leq 4$. Finally,  we discuss their derivations.

%
%The main purpose of this paper is to establish elementary properties of
%BiHom-diassociative (resp. BiHom-Poisson Dialgebras) and central extensions of BiHom-diassociative. Then, we discuss their 2-cocycles.
 %We provide a classification of $n$-dimensional BiHom-diassociative algebras for $n\leq 4$. Furthermore,  we discuss their derivations.

\noindent \hrulefill

\vspace{.3in}

 \noindent {\bf AMS Subject Classification: } .

\vspace{.08in} \noindent \textbf{Keywords}: 
 BiHom-associative dialgebra, BiHom-Poisson dialgebra, centroid, averaging operator, Nijenhuis operator, Rota-Baxter operator,
 Extension, Derivation, Classification.
\vspace{.3in}
% \noindent {\small Revised: April 16, 2014, June 28, 2014 $\parallel$  Accepted: July 4, 2014}
\vspace{.2in}

% ------------ [Running Heads - for odd and even pages] - please insert it only on page 2!
\pagestyle{fancy} \fancyhead{} \fancyhead[EC]{ } 
\fancyhead[EL,OR]{\thepage} \fancyhead[OC]{Ahmed Zahari and Ibrahima Bakayoko} \fancyfoot{}
\renewcommand\headrulewidth{0.5pt}
%------------------------------------------------------------------------------

\section{Introduction}
The associative dialgebras (also known as diassociative algebras) has been introduced by Loday in 1990 (see [6] and references therein) as  a
 generalization of associative algebras.
They are a generalization of associative algebras in the sens that they possess two associative multiplications and obey to three other conditions; 
when the two associative low are equal we recover associative algebra.
One of his motivation were to find an algebra whose commutator give rises to Leibniz algebra as it is the case in the relation between Lie and 
associative algebra. Another motivation come from the research of an obstruction to the periodicity in algebraic
K-theory. Now, these algebras found their applications in classical geometry, non-commutative geometry and physics.

The centroid plays an important role in understanding forms of an algebra. It is an element in the classification of associative and diassociative
algebras. They occurs naturally is in the study of derivations of an algebra. The centroid  and averaging operators are used in the deformation of
 algebra in order to
generate another algebraic structure. The Nijenhuis operator on an associative algebra was introduced in \cite{CJ} to study quantum bi-Hamiltonian
systems while the notion Nijenhuis operator on a Lie algebra originated from the concept of Nijenhuis tensor that was introduced by Nijenhuis
in the study of pseudo-complex manifolds and was related to the well known concepts of Schouten-Nijenhuis bracket , the Frolicher-Nijenhuis
 bracket \cite{FN}, and the Nijenhuis-Richardson bracket. The associative analog of the Nijenhuis relation may be regaded as the homogeneous version of 
Rota-Baxter relation\cite{PL}.

BiHom-algebraic structures were introduced  in 2015 by G. Graziani, A. Makhlouf, C. Menini and F. Panaite 
 in \cite{GACF} from a categorical approach  as an extension of the class of Hom-algebras.
Since then, other interesting BiHom-type algebraic structures of many Hom-algebraic structures has been intensively studied as 
BiHom-Lie colour algebras structures \cite{ABA}, Representations of BiHom-Lie algebras \cite{YH}, 
BiHom-Lie superalgebra structures \cite{SS}, 
$\{\sigma, \tau\}$-Rota-Baxter operators, infinitesimal
Hom-bialgebras and the associative (Bi)Hom-Yang-Baxter equation \cite{MFP}, The construction and deformation of BiHom-Novikov algebras \cite{SZW},
On n-ary Generalization of BiHom-Lie algebras
and BiHom-Associative Algebras \cite{KMS}, Rota-Baxter operators on BiHom-associative
algebras and related structures \cite{LAC}.

% A  BiHom-associative algebra $(A, \mu,\lambda ,\alpha, \beta)$ is consisting of a vector space, two multiplications and a linear self map 
% (resp. two multiplications and two linear self maps ). It may be viewed as a deformation of an associative algebras, in which the associativity
%  condition is twisted by a linear map $\alpha$ and $\beta$  and such that when $\alpha=id$ (rep. $\beta=id$), the BiHom-diassociative algebras
%  degenerate to exactly an diassociative algebras. 

The goal of this paper is to introduce, classify and study structures, central extensions and derivations of BiHom-associative algebras. The paper
is organized as follows. In section 2, we define BiHom-associative dialgebras, give some constructions using twisting,
direct sum, elements of centroid, averaging operator, Nijenhuis operator
and Rota-Baxter relation. We give a connection between BiHom-associative dialgebras and BiHom-Leibniz algebras.
We introduce action of a BiHom-Leibniz algebra onto another and give a Leibniz structure on the semidirect structure. Then, we show that
the semidirect sum of BiHom-Leibniz algebras associated to BiHom-associative dialgebras is the same that the BiHom-Leibniz algebra associated
to the semidirect of BiHom-associative dialgebras. Finally, we introduce BiHom-associative dialgebras and show that any BiHom-associative
dialgebra carries a structure of BiHom-Poisson dialgebra.
In section 3, we introduce the notion of central extension of BiHom-associative dialgebras and define $2$-cocycles and $2$-coboundaries
of BiHom-associative dialgebras with coefficients in a trivial BiHom-module. Then we establish relationship between $2$-cocycles and central 
extensions. Section 4, is devoted to the classification of $n$-dimensional BiHom-associative dialgebras for $n\leq 4$. We dedicated Section 5 to the 
derivations of BiHom-associative dialgebras.

\section{Structure of BiHom-associative dialgebras}
\begin{definition}\label{dia}
A BiHom-associative dialgebras is a $5$-truple $(A, \dashv, \vdash, \alpha, \beta)$ consisting of a  linear space $A$  linear maps
 $\dashv, \vdash,: A\times A \longrightarrow A$ and  $\alpha, \beta : A\longrightarrow A$ satisfying, for all $x, y, z\in A$ the following
 conditions : 
\begin{eqnarray}
\alpha\circ\beta&=&\beta\circ\alpha,\\
(x\dashv y)\dashv\beta(z)&=&\alpha(x)\dashv(y\dashv z),\label{eq4}\\
(x\dashv y)\dashv\beta(z)&=&\alpha(x)\dashv(y\vdash z),\label{eq5}\\
(x\vdash y)\dashv\beta(z)&=&\alpha(x)\vdash(y\dashv z),\label{eq6}\\
(x\dashv y)\vdash\beta(z)&=&\alpha(x)\vdash(y\vdash z),\label{eq7}\\
(x\vdash y)\vdash\beta(z)&=&\alpha(x)\vdash(y\vdash z).\label{eq8}
\end{eqnarray}
\end{definition}
We called $\alpha$ and $\beta$ ( in this order ) the structure maps of A.

\begin{example}
Any Hom-associative dialgebra \cite{BB2} or any associative dialgebra is a BiHom-associative dialgebra  by 
setting $\beta = \alpha$ or $\alpha=\beta=id$.
\end{example}

\begin{example}
 Let $(A, \dashv, \vdash, \alpha, \beta)$ a BiHom-associative dialgebra. Consider the module of $n\times n$-matrices 
$\mathcal{M}_n(D)=\mathcal{M}_n(\mathbb{K})\otimes D$ with the linear maps ${\bf \alpha} (A)=(\alpha(a_{ij}))$, ${\bf\beta}(A)=(\beta(a_{ij})$ for all 
$A\in \mathcal{M}_n(D)$ and the products $(a\triangleleft b)_{ij}=\sum_{k}a_{ik}\dashv b_{kj}$ and $(a\triangleright b)_{ij}=\sum_{k}a_{ik}\vdash b_{kj}$.
Then, $(\mathcal{M}_n(D), \triangleleft, \triangleright, {\bf \alpha}, {\bf \beta})$ is a BiHom-associative dialgebra. 
\end{example}

\begin{definition}
 A morphism $ f : ({D}, \dashv, \vdash, \alpha, \beta)$ and $({D}', \dashv',\vdash', \alpha', \beta')$ be a BiHom-associative dialgebras is a linear map 
$f : {D}\rightarrow {D}'$ such that $\alpha'\circ f=f\circ\alpha,\, \beta'\circ f=f\circ\beta$ and 
$f(x\dashv y)=f(x)\dashv'f(y),\quad f(x\vdash y)=f(x)\vdash'f(y)$, for all
 $x, y \in {D}.$
\end{definition}

\begin{definition}
A BiHom-associative dialgebra $(A, \dashv, \vdash, \alpha, \beta)$ in which $\alpha$ and $\beta$ are morphism is said to be a multiplicative
 BiHom-associative dialgebra.\\
If moreover, $\alpha$ and $\beta$ are bijective (i.e. automorphisms), then $(A, \dashv, \vdash, \alpha, \beta)$ is said to be a 
regular BiHom-associative dialgebra.
\end{definition}

We prove in the following proposition  that any  BiHom-associative dialgebra turn to another one via morphisms.
\begin{theorem}\label{tw}
 Let $(D, \dashv, \vdash, \alpha, \beta )$ be a  BiHom-associative dialgebra and $\alpha', \beta' : D\rightarrow D$ two morphisms of 
BiHom-associative dialgebras such that the maps $\alpha, \alpha', \beta, \beta'$ commute pairewise. Then
$$D_{(\alpha', \beta')}=(D, \triangleleft:=\dashv(\alpha'\otimes\beta'), \triangleright:=\vdash(\alpha'\otimes\beta'), \alpha\alpha', \beta\beta')$$
is a  BiHom-associative dialgebra. 
\end{theorem}
\begin{proof}
We prove only one axiom and leave the rest to the reader. For any $x, y, z\in D$,
 \begin{eqnarray}
(x\triangleleft y)\triangleleft\beta\beta'(z)- \alpha\alpha'(x)\triangleleft(y\triangleright z)&=&
\alpha'(\alpha'(x)\dashv\beta'(y))\dashv\beta'\beta\beta'(z)-\alpha'\alpha\alpha'(x)\dashv\beta'(\alpha'(y)\vdash\beta'(z))\nonumber\\
&=&(\alpha'\alpha'(x)\dashv\alpha'\beta'(y))\dashv\beta\beta'\beta'(z)-\alpha\alpha'\alpha'(x)\dashv(\alpha'\beta'(y)\vdash\beta'\beta'(z)).\nonumber
 \end{eqnarray}
The left hand side vanishes by (\ref{eq5}). And, this ends the proof.
\end{proof}

\begin{corollary}
 Let $(D, \dashv, \vdash, \alpha, \beta )$ be a multiplicative BiHom-associative dialgebra. Then
$$(D, \dashv\circ(\alpha^n\otimes\beta^n), \vdash\circ(\alpha^n\otimes\beta^n), \alpha^{n+1}, \beta^{n+1})$$
is also a multiplicative  BiHom-associative dialgebra.
\end{corollary}
\begin{proof}
 It suffises to take $\alpha'=\alpha^n$ and $\beta'=\beta^n$ in Theorem \ref{tw}.
\end{proof}

% By taking $\alpha'=\beta'=\alpha\neq id$ in Theorem \ref{tw}, we get the following corollary.

 \begin{corollary}
 Let $(D, \dashv, \vdash, \alpha)$ be a multiplicative Hom-associative dialgebra and $\beta : D\rightarrow D$ an endomorphism of $D$. Then
$$(D, \dashv\circ(\alpha\otimes\beta), \vdash\circ(\alpha\otimes\beta), \alpha^{2}, \beta)$$
is also a Hom-associative dialgebra.
\end{corollary}
\begin{proof}
 It suffises to take $\alpha'=\alpha$ and replace $\beta$ by $Id_D$, and $\beta'$ by $\beta$ in Theorem \ref{tw}.
\end{proof}

Any regular Hom-associative dialgebra give rises to associative dialgebra as stated in the next corollary.
\begin{corollary}
If $(D, \dashv, \vdash, \alpha, \beta)$ is a regular BiHom-associative dialgebra, then
$$(D, \dashv\circ(\alpha^{-1}\otimes\beta^{-1}), \vdash\circ(\alpha^{-1}\otimes\beta^{-1}))$$
is an associative dialgebra.
\end{corollary}

\begin{proof}
 We have to take $\alpha'=\alpha^{-1}$ and $\beta'=\beta^{-1}$ in Theorem \ref{tw}.
\end{proof}

 \begin{corollary}
 Let $(D, \dashv, \vdash)$ be an associative dialgebra and $\alpha : D\rightarrow D$ and $\beta : D\rightarrow D$ a pair of
 commuting endomorphisms of $D$. Then
$$(D, \dashv\circ(\alpha\otimes\beta), \vdash\circ(\alpha\otimes\beta), \alpha, \beta)$$
is a  BiHom-associative dialgebra.
\end{corollary}
\begin{proof}
 We have to take $\alpha=\beta=Id_D$ and replace $\alpha'$ by $\alpha$, and $\beta'$ by $\beta$ in Theorem \ref{tw}.
\end{proof}

\begin{definition}
Let $(D, \dashv, \vdash, \alpha, \beta)$ be a  BiHom-associative dialgebra. For any integers $k, l$,  an even linear map $\theta : D\rightarrow D$
 is called an element of $(\alpha^k, \beta^l)$-centroid on $D$ if
\begin{eqnarray}
\alpha\circ\theta&=&\theta\circ\alpha,\quad \beta\circ\theta=\theta\circ\beta,\\
 \theta(x)\dashv \alpha^k\beta^l(y)&=&\theta(x)\dashv\theta(y)=\alpha^k\beta^l(x)\dashv \theta(y),\\ 
 \theta(x)\vdash \alpha^k\beta^l(y)&=&\theta(x)\vdash \theta(y)=\alpha^k\beta^l(x)\vdash \theta(y),
\end{eqnarray}
for all $x, y\in D$.
 \end{definition}
The set of elements of centroid is called centroid.

\begin{proposition}\label{t1}
 Let $(A, \dashv, \vdash, \alpha, \beta)$ a BiHom-associative dialgebra and $\phi : A\rightarrow A$ and $\psi : A\rightarrow A$ be a paire of commuting
elements of cenroid. Let us defined
$$x\triangleleft y:=\phi(x)\dashv y \quad\mbox{and}\quad x\triangleright y:=\psi(x)\vdash y.$$
Then, $(A, \triangleleft, \triangleright, \alpha, \beta)$ is a BiHom-associative dialgebra if and only if\\
$Im(\phi-\psi)\in Z_{\dashv}(A):=\{x\in A/x\dashv y=0, \forall y\in A\}
 \;\mbox{and}\; 
Im(\phi-\psi)\in Z_{\vdash}(A):=\{x\in A/y\vdash x=0, \forall y\in A\}$.
\end{proposition}
\begin{proof}
 We only prove axioms (\ref{eq5}) and (\ref{eq7}), the three other comes from BiHom-associativity.
So for any $x, y, z\in A$,
\begin{eqnarray}
 (x\triangleleft y)\triangleleft\beta(z)-\alpha(x)\triangleleft(y\triangleright z)
&=&(\phi(x)\dashv y)\dashv\phi\beta(z)-\phi\alpha(x)\dashv(y\vdash\psi(z))\nonumber\\
&=&(\phi(x)\dashv y)\dashv\beta\phi(z)-\alpha\phi(x)\dashv(y\vdash\psi(z))\nonumber\\
&=&(\phi(x)\dashv y)\vdash\beta\phi(z)-(\phi(x)\dashv y)\vdash\beta\psi(z))\nonumber\\
&=&(\phi(x)\dashv y)\vdash\beta(\phi-\psi)(z)\nonumber\\
&=&\alpha\phi(x)\dashv (y\vdash(\phi-\psi)(z))\nonumber.
\end{eqnarray}
and
\begin{eqnarray}
 (x\triangleleft y)\triangleright \beta(z)-\alpha(x)\triangleright(y\triangleright z)
&=&(\phi(x)\dashv y)\vdash\beta\psi(z)-\psi\alpha(x)\vdash(y\vdash\psi(z))\nonumber\\
&=&(\phi(x)\dashv y)\vdash\beta\psi(z)-\alpha\psi(x)\vdash(y\psi(z))\nonumber\\
&=&(\phi(x)\dashv y)\vdash\beta\psi(z)-(\psi(x)\dashv y)\vdash\beta\psi(z)\nonumber\\
&=&[(\phi(x)-\psi(x))\dashv y]\vdash\beta\psi(z)\nonumber.
\end{eqnarray}
A study of cancellation of the two equalities allows to conclude.
\end{proof}

\begin{proposition}\label{t2}
 Let $(A, \cdot, \alpha, \beta)$ be a BiHom-associative algebra, and $\phi : A\rightarrow A$ and $\psi : A\rightarrow A$ be a paire of commuting
elements of cenroid. Let us defined
$$x\dashv y:=\phi(x)\cdot y \quad\mbox{and}\quad x\vdash y:=\psi(x)\cdot y.$$
Then, $(A, \dashv, \vdash, \alpha, \beta)$ is a BiHom-associative dialgebra if and only if $Im(\phi-\psi)$ is contained in the 
set of isotropic vectors.
\end{proposition}
\begin{proof}
 We only prove axioms (\ref{eq5}) and (\ref{eq7}), the three other comes from BiHom-associativity.
So for any $x, y, z\in A$,
\begin{eqnarray}
 (x\dashv y)\dashv\beta(z)-\alpha(x)\dashv(y\vdash z)
&=&(\phi(x)y)\phi\beta(z)-\phi\alpha(x)(y\psi(z))\nonumber\\
&=&(\phi(x)y)\beta\phi(z)-\alpha\phi(x)(y\psi(z))\nonumber\\
&=&(\phi(x)y)\beta\phi(z)-(\phi(x)y)\beta\psi(z))\nonumber\\
&=&(\phi(x)y)\beta(\phi-\psi)(z)\nonumber\\
&=&\alpha\phi(x)(y(\phi-\psi)(z))\nonumber.
\end{eqnarray}
and
\begin{eqnarray}
 (x\dashv y)\vdash \beta(z)-\alpha(x)\vdash(y\vdash z)
&=&(\phi(x)y)\beta\psi(z)-\psi\alpha(x)(y\psi(z))\nonumber\\
&=&(\phi(x)y)\beta\psi(z)-\alpha\psi(x)(y\psi(z))\nonumber\\
&=&(\phi(x)y)\beta\psi(z)-(\psi(x)y)\beta\psi(z)\nonumber\\
&=&[(\phi(x)-\psi(x))y]\beta\psi(z)\nonumber.
\end{eqnarray}
A study of cancellation of the two equalities allow to conclude.
\end{proof}

\begin{remark}
 Proposition \ref{t2} may be seen as a consequence of Proposition \ref{t1}.
\end{remark}

\begin{proposition}
 Let $(A, \cdot, \alpha, \beta)$ be a BiHom-associative algebra and $(M, \ast_L, \ast_R, \alpha_M, \beta_M)$ an $A$-BiHom-bimodule i.e. $M$ is a
 vector space, $\alpha_M :M\rightarrow M$ and $\beta_M : M\rightarrow M$ are two linear maps, and $\ast_L : A\rightarrow M$ and
 $\ast_R : M\rightarrow A$ two bilinear maps such that
\begin{eqnarray}
 \alpha(x)\ast_L(y\ast_L m)&=&(x\cdot y)\ast_L\beta_M(m)\\
\alpha(x)\ast_L(m\ast_R y)&=&(x\ast_L m)\ast_R\beta(y)\label{m2}\\
\alpha_M(m)\ast_R(x\cdot y)&=&(m\ast_R x)\ast_R\beta(y).
\end{eqnarray}
Suppose that $f :M\rightarrow A$ is a morphism of $A$-BiHom-bimodule i.e. $f$ is linear such that $\alpha\circ f=f\circ\alpha_M$,
$\beta\circ f=f\circ\beta_M$ and
\begin{eqnarray}
 f(x\ast_L m)&=&x\cdot f(m)\\
f(m\ast_R x)&=&f(m)\cdot x.\\
\end{eqnarray}
Then, $(M, \triangleleft, \triangleright, \alpha_M, \beta_M)$ is a BiHom-associative dialgebra with
$$m\triangleleft n=f(m)\ast_R n\quad\mbox{and}\quad m\triangleright n=m\ast_Rf(n),$$
for all $m, n\in M$.
\end{proposition}
\begin{proof}
 We only prove axiom (\ref{eq8}), the other being proved similarly. For any $x, y, z\in A$,
\begin{eqnarray}
 (m\triangleleft n)\triangleright \beta_M(p)
&=&(f(m)\ast_L n)\ast_R f\beta_M(p)\nonumber\\
&=&(f(m)\ast_L n)\ast_R \beta f(p)\nonumber.
\end{eqnarray}
By (\ref{m2}),
\begin{eqnarray}
 (m\triangleleft n)\triangleright \beta_M(p)
&=&\alpha f(m)\ast_L (n\ast_R f(p))\nonumber\\
&=& f\alpha_M(m)\ast_L (n\triangleright p)\nonumber\\
&=& \alpha_M(m)\triangleleft(n\triangleright p)\nonumber.
\end{eqnarray}
This completes the proof.
\end{proof}
\begin{remark}
 Any $(\alpha^0, \beta^0)$-element of centroid of a BiHom-associative algebra is a morphism of BiHom-bimodule. 
\end{remark}
Thanks to the above remark, we have what follows :
\begin{corollary}
 Let $(A, \cdot, \alpha, \beta)$ be a BiHom-associative algebra and let $\theta$ be an element of cenroid on $A$. Then, 
$(A, \triangleleft, \triangleright, \alpha, \beta)$ is a BiHom-associative dialgebra with 
$$x\triangleleft y=\theta(x)\cdot y\quad\mbox{and}\quad x\triangleright y=x\cdot \theta(y),$$
for any $x, y\in A.$
\end{corollary}

\begin{proposition}
 Let $(D, \dashv, \vdash, \alpha, \beta)$ be a BiHom-associative dialgebra and $R: D\rightarrow D$ a Rota-Baxter operator of weight $0$ on $D$ i.e.
 $R$ is linear and $\alpha\circ R=R\circ\alpha$ , $\beta\circ R=R\circ\beta$, and
\begin{eqnarray}
 R(x)\dashv R(y)&=&R(R(x)\dashv y+x\dashv R(y))\\
R(x)\vdash R(y)&=&R(R(x)\vdash y+x\vdash R(y))
\end{eqnarray}
Then, $(D, \triangleleft, \triangleright, \alpha, \beta)$ is also a BiHom-associative algebra with
\begin{eqnarray}
 x\triangleleft y=R(x)\dashv y+x\dashv R(y),\\
x\triangleright y=R(x)\vdash y+x\vdash R(y),
\end{eqnarray}
for all $x, y\in D$.
\end{proposition}
\begin{proof}
We only prove axiom (\ref{eq8}), the other being proved in a similar way. Thus, For any $x, y, z\in A$,
 \begin{eqnarray}
&&\qquad  (x\triangleleft y)\triangleleft \beta(z)-\alpha(x)\triangleleft(y\triangleright z)=\nonumber\\
&&=(x\dashv R(y)+R(x)\dashv y)\dashv R\beta(z)+R(R(x)\dashv y+x\dashv R(y))\dashv\beta(z)\nonumber\\
&&\quad-\alpha(x)\dashv R(R(y)\dashv z+y\dashv R(z))-R\alpha(x)\dashv(R(y)\vdash z+y\vdash R(z))\nonumber\\
&&=(x\dashv R(y))\dashv \beta R(z)+(R(x)\dashv y)\dashv\beta R(z)+(R(x)\dashv R(y))\dashv \beta(z)\nonumber\\
&&\quad-\alpha(x)\dashv(R(y)\vdash R(z))-\alpha R(x)\dashv(y\vdash R(z))\alpha R(x)\dashv(R(y)\vdash z).\nonumber
 \end{eqnarray}
The left hand side vanishes by axiom (\ref{eq8}). This ends the proof.
\end{proof}

\begin{corollary}
  Let $(D, \dashv, \vdash, \alpha, \beta)$ BiHom-associative dialgebra and $R: D\rightarrow D$ a Rota-Baxter operator of weight $0$ on $D$. Then,
$(D, \ast, \alpha, \beta)$ is a BiHom-associative algebra with $x\ast y=x\triangleleft y+x\triangleright y$.
\end{corollary}

\begin{corollary}
  Let $(D, \dashv, \vdash, \alpha, \beta)$ be a BiHom-associative dialgebra and $R: D\rightarrow D$ a Rota-Baxter operator of weight $0$ on $D$. Then,
$(D, [-, -], \alpha, \beta)$ is a BiHom-Lie algebra with $$[x, y]=x\ast y-\alpha^{-1}\beta(y)\ast\alpha\beta^{-1}(x),$$
with $x\ast y=x\triangleleft y+x\triangleright y$.
\end{corollary}

As in the previous proposition, it is well known that a Nijenhuis operator on an associative algebra allows to define another associative algebra.
In the next result, we try to establish an analoq of this result for BiHom-associative dialgebras.
\begin{proposition}\label{tns}
 Let $(D, \dashv, \vdash, \alpha, \beta)$ BiHom-associative dialgebra and $N: D\rightarrow D$ a Nijenhuis operator on $D$ i.e.
 $N$ is linear and $\alpha\circ N=N\circ\alpha$ , $\beta\circ N=N\circ\beta$, and
\begin{eqnarray}
 N(x)\dashv N(y)&=&N(N(x)\dashv y+x\dashv N(y)-N(x\cdot y))\label{n1}\\
N(x)\vdash N(y)&=&N(N(x)\vdash y+x\vdash N(y)-N(x\cdot y))\label{n2}
\end{eqnarray}
Then, $(D, \triangleleft, \triangleright, \alpha, \beta)$ is also a BiHom-associative algebra with
\begin{eqnarray}
 x\triangleleft y=N(x)\dashv y+x\dashv N(y)-N(x\dashv y),\\
x\triangleright y=N(x)\vdash y+x\vdash N(y)-N(x\vdash y),
\end{eqnarray}
for all $x, y\in D$.
\end{proposition}
\begin{proof}
We only prove axiom (\ref{eq6}) for the products $\triangleleft$ and $\triangleright$. The others are leave to the reader. 
 \begin{eqnarray}
&&\qquad  (x\triangleright y)\triangleleft \beta(z)-\alpha(x)\triangleright(y\triangleleft z)=\nonumber\\
&&=N\Big(N(x)\vdash y+x\dashv y-N(x\dashv y)\Big)\dashv \beta(z)+\Big(N(x)\dashv y+x\vdash N(y)-N(x\vdash y)\Big)\dashv N\beta(z)\nonumber\\
&&\quad-N\Big((N(x)\vdash y+x\vdash N(y)-N(x\vdash y))\dashv\beta(z)\Big)
-N\alpha(x)\vdash\Big(N(y)\dashv z+y\dashv N(z)-N(y\dashv z)\Big)\nonumber\\
&&-\alpha(x)\vdash N\Big(N(y)\dashv z+y\dashv N(z)-N(y\dashv z)\Big)
+N\Big(\alpha(x)\vdash(N(y)\dashv z+y\dashv N(z)-N(y\dashv z))\Big)\nonumber.
 \end{eqnarray}
By (\ref{n1}) and  (\ref{n2}), we have
\begin{eqnarray}
&&\qquad  (x\triangleright y)\triangleleft \beta(z)-\alpha(x)\triangleright(y\triangleleft z)=\nonumber\\
&&=(N(x)\vdash N(y))\dashv \beta(z)+(N(x)\dashv y)\dashv \beta N(z)+(x\vdash N(y))\dashv \beta N(z)-N(x\vdash y))\dashv N\beta(z)\nonumber\\
&&\quad-N\Big((N(x)\vdash y)\dashv\beta(z)\Big)-N\Big((x\vdash N(y))\dashv\beta(z)\Big)-N\Big(N(x\vdash y)\dashv\beta(z)\Big)\nonumber\\
&&-\alpha N(x)\vdash(N(y)\dashv z)-\alpha N(x)\vdash(y\dashv N(z))+N\alpha(x)\vdash N(y\dashv z))\nonumber\\
&&-\alpha(x)\vdash (N(y)\dashv N(z))
+N\Big(\alpha(x)\vdash(N(y)\dashv z\Big)+N\Big(\alpha(x)\vdash(y\dashv N(z))\Big)-N\Big(\alpha(x)\vdash N(y\dashv z))\Big)\nonumber.
 \end{eqnarray}
By (\ref{eq6}), we have
\begin{eqnarray}
&&\qquad  (x\triangleright y)\triangleleft \beta(z)-\alpha(x)\triangleright(y\triangleleft z)=\nonumber\\
&&=-N(x\vdash y))\dashv N\beta(z)-N\Big((N(x)\vdash y)\dashv\beta(z)\Big)-N\Big((x\vdash N(y))\dashv\beta(z)\Big)-N\Big(N(x\vdash y)\dashv\beta(z)\Big)\nonumber\\
&&\quad+N\alpha(x)\vdash N(y\dashv z))
+N\Big(\alpha(x)\vdash(N(y)\dashv z\Big)+N\Big(\alpha(x)\vdash(y\dashv N(z))\Big)-N\Big(\alpha(x)\vdash N(y\dashv z))\Big)\nonumber.
 \end{eqnarray}
Using again (\ref{n1}) and  (\ref{n2}), it comes
\begin{eqnarray}
&&\qquad  (x\triangleright y)\triangleleft \beta(z)-\alpha(x)\triangleright(y\triangleleft z)=\nonumber\\
&&=-N\Big(N(x\vdash y)\dashv\beta(z)+ (x\vdash y)\dashv\beta N(z)-N((x\vdash y)\dashv\beta(z))\Big)\nonumber\\
&&-N\Big((N(x)\vdash y)\dashv\beta(z)\Big)-N\Big((x\vdash N(y))\dashv\beta(z)\Big)-N\Big(N(x\vdash y)\dashv\beta(z)\Big)\nonumber\\
&&\quad+N\Big(\alpha(x)\vdash (N(y)\dashv z)+\alpha(x)\vdash N(y\dashv z)-N(\alpha(x)\vdash (y\dashv z))\Big)\nonumber\\
&&+N\Big(\alpha(x)\vdash(N(y)\dashv z\Big)+N\Big(\alpha(x)\vdash(y\dashv N(z))\Big)-N\Big(\alpha(x)\vdash N(y\dashv z))\Big)\nonumber.
 \end{eqnarray}
The left hand side vanishes by (\ref{eq6}).
\end{proof}

\begin{corollary}
 If $(D, \dashv, \vdash, \alpha)$ is a Hom-associative dialgebra and $N: D\rightarrow D$ a Nijenhuis operator on $D$, then 
$(D, \triangleleft, \triangleright, \alpha)$ is also a Hom-associative algebra with
\begin{eqnarray}
 x\triangleleft y=N(x)\dashv y+x\dashv N(y)-N(x\dashv y),\nonumber\\
x\triangleright y=N(x)\vdash y+x\vdash N(y)-N(x\vdash y),\nonumber
\end{eqnarray}
for all $x, y\in D$.
\end{corollary}

\begin{corollary}
 If $(D, \dashv, \vdash, \alpha, \beta)$ is an associative dialgebra and $N: D\rightarrow D$ a Nijenhuis operator on $D$, then 
$(D, \triangleleft, \triangleright, \alpha, \beta)$ is also an associative algebra with
\begin{eqnarray}
 x\triangleleft y=N(x)\dashv y+x\dashv N(y)-N(x\dashv y),\nonumber\\
x\triangleright y=N(x)\vdash y+x\vdash N(y)-N(x\vdash y),\nonumber
\end{eqnarray}
for all $x, y\in D$.
\end{corollary}

The next proposition asserts that the twist of the products of any BiHom-associative dialgebra by an averaging operator gives rise to another
BiHom-associative dialgebra.

\begin{proposition}
 Let $(D, \dashv, \vdash, \alpha, \beta)$ be a BiHom-associative dialgebra and $\theta: D\rightarrow D$ an injective averaging operator on $D$ i.e.
 $\theta$ is an injective linear map such that $\alpha\circ \theta=\theta\circ\alpha$ , $\beta\circ \theta=\theta\circ\beta$, and
\begin{eqnarray}
 \theta(x)\dashv\theta(y)&=&\theta(\alpha^k\beta^l(x)\dashv\theta(y))=\theta(\theta(x)\dashv\alpha^k\beta^l(y)),\label{c1}\\
\theta(x)\vdash\theta(y)&=&\theta(\alpha^k\beta^l(x)\vdash\theta(y))=\theta(\theta(x)\vdash\alpha^k\beta^l(y)),\label{c2}
\end{eqnarray}
for any $x, y\in D$.
Then, $(D, \triangleleft, \triangleright, \alpha, \beta)$ is also a BiHom-associative algebra with
\begin{eqnarray}
 x\triangleleft y=\theta(x)\dashv\alpha^k\beta^l(y))\\
x\triangleright y=\alpha^k\beta^l(x)\vdash\theta(y),
\end{eqnarray}
for all $x, y\in D$.
\end{proposition}
\begin{proof}
We only prove one identity, the others have a similar proof. For any $x, y, z\in D$, one has :
 \begin{eqnarray}
&&\qquad  \theta[(x\triangleleft y)\triangleright\beta(z)-\alpha(x)\triangleright(y\triangleleft z)]=\nonumber\\
&&=\theta[\theta(\theta(x)\dashv\alpha^k\beta^l(y)))\vdash\alpha^k\beta^{l+1}(z))
-\theta\alpha(x)\vdash(\theta(y)\dashv\alpha^k\beta^l(z)))]\nonumber\\
&&=\theta[(\theta(x)\dashv\theta(y)\vdash\alpha^k\beta^{l+1}(z)]
-\theta\alpha(x)\vdash\theta(\theta(y)\dashv\alpha^k\beta^l(z)))]\nonumber\\
&&=(\theta(x)\vdash\theta(y))\dashv\beta\theta(z)-\alpha\theta(x)\vdash(\theta(y)\dashv\theta(z))\nonumber.
 \end{eqnarray}
Which vanishes by axiom (\ref{eq6}), and the conclusion holds by injectivity.
\end{proof}

At this moment, we introduce ideals for BiHom-associative dialgebra in order to give another construction of BiHom-associative dialgebras.

\begin{definition}
Let $(D, \dashv, \vdash, \alpha, \beta)$ be a BiHom-associative dialgebra and $D_o$ a subset of $D$. We say that $D_o$ is a BiHom-subalgebra of $D$ if
$D_o$ is stable under $\alpha$ and $\beta$, and $x\dashv y, x\vdash y\in D_o$, for any $x, y\in D_o$. 
\end{definition}
\begin{example}
  If $\varphi : D_1\rightarrow D_2$ is a homomorphism of BiHom-associative dialgebras, the image $Im\varphi$ is a BiHom-subalgebra of  $D_2$.
\end{example}

\begin{definition}
A two side BiHom-ideal of a BiHom-associative dialgebra $(D, \dashv, \vdash, \alpha, \beta)$ is subspace $I$ such that $\alpha(I)\subset I, 
 x\ast y, y\ast x\in I$ for all $x\in D, y\in I$ with $\ast=\dashv$ and $\vdash$. Note that $I$ is called the left and right BiHom-ideal if $x\dashv y, x\vdash y$ and 
$y\dashv x, y\vdash x$ are in $I$, respectively, for all $x\in D. y\in I$.
\end{definition}

\begin{example}
 i) Obviously $I=\{0\}$ and $I=D$ are two-sided ideals.\\
ii) If $\varphi : D_1\rightarrow D_2$ is a homomorphism of BiHom-associative dialgebras, the kernel $Ker\varphi$ is a two sided ideal in $D_1$.\\
iii) If $I_1$ and $I_2$ are two sided ideals of $D$, then so is $I_1+I_2$.
\end{example}

In the below proposition, we prove that BiHom-associative dialgebras are closed under direct summation, and give a condition for which
a linear map becomes a morphism.

\begin{proposition}
Let $({A},  \dashv_{A}, \vdash_{A}, \alpha_{A}, \beta_{A})$ and $({B}, \dashv_{B}, \vdash_{B},  \alpha_{B}, \beta_{B})$ be two 
 BiHom-associative dialgebras. Then there exists a BiHom-associative dialgebra structure
on ${A}\oplus{B}$ with the bilinear maps $\triangleleft, \triangleright : ({A}\oplus{B})^{\otimes 2}\rightarrow {A}\oplus{B}$
  given by 
$$(a_1+b_2)\dashv(a_2+b_2)=a_1\dashv_{A}a_2+b_2\dashv_{B}b_2,$$ $$(a_1+b_1)\vdash(a_2+b_2)=a_1\vdash_{A}a_2+b_{A}\vdash_{B}b_2$$ and the linear maps 
$\alpha=\alpha_{A}+\alpha_{B},\, \beta=\beta_{A}+\beta_{B} : {A}\oplus{B}\rightarrow {A}\oplus{B}$ given by 
$$(\alpha_{A}+\alpha_{B})(a+b)=\alpha_{A}(a)+\alpha_{B}(b),\, (\beta_{A}+\beta_{B})(a+b)=\beta_{A}(a)+\beta_{B}(b),\, 
\forall(a,b)\in({A}\times{B}).  
$$
Moreover, if $\xi : {A}\rightarrow {B}$  is a linear map. 
Then $ \xi : ({A}, \dashv_{A}, \vdash_{A}, \alpha_A, \beta_A)$ to  
$({B}, \dashv_{B}, \vdash_{B},  \alpha_B, \beta_B)$ is a morphism if and only if its graph  $\Gamma_\xi=\{(x, \xi(x)), x\in A\}$
 is a BiHom-subalgebra of $({A}\oplus{B}, \triangleleft, \triangleright, \alpha, \beta)$.
\end{proposition}

\begin{proof}
The proof of the first part of the proposition comes from a simple computation.\\
Let us suppose that $\xi : ({A}, \dashv_{A}, \vdash_{A},\alpha_A, \beta_A)\rightarrow({B}, \dashv_{B}, \vdash_{B},\alpha_B, \beta_B)$ is a
morphism of BiHom-associative dialgebras.
Then $$(u+\xi(u))\dashv(v+\xi(v))=(u\dashv_{A}v+\xi(u)\dashv_{B}\xi(v))=(u\dashv_{A}v+\xi(u\dashv_{A}v)$$
     $$(u+\xi(u))\vdash(v+\xi(v))=(u\vdash_{A}v+\xi(u)\vdash_{B}\xi(v))=(u\vdash_{A}v+\xi(u\vdash_{A}v).$$
Thus the graph $\Gamma_\xi$ is closed under the operations $\dashv$ and $\vdash$.\\
 Furthermore since $\xi\circ\alpha_A=\alpha_B\circ\xi,$ and $\xi\circ\beta_A=\beta_B\circ\xi,$ we have 
$$
(\alpha_A\oplus\alpha_B)(u, \xi(u))=(\alpha_A(u), \alpha_B\circ\xi(u))=(\alpha_A(u), \xi\circ\alpha_A(u)).
$$
and 
$$
(\beta_A\oplus\beta_B)(u, \xi(u))=(\beta_A(u), \beta_B\circ\xi(u))=(\beta_A(u), \xi\circ\beta_A(u)),
$$
implies that  $\Gamma_\xi$ is closed $\alpha_A\oplus\alpha_B$ and $\beta_A\oplus\beta_B.$
 Thus, $\Gamma_\xi$  is a BiHom-subalgebra of 
$({A}\otimes{B}, \dashv, \vdash, \alpha, \beta).$ \\
Conversely, if the graph $\Gamma_\xi\subset{A}\oplus{B}$ is a BiHom-subalgebra of 
$({A}\oplus{B}, \dashv, \vdash, \alpha, \beta)$ then we 
$$(u+\xi(u))\dashv(v+\xi(v))=(u\dashv_{A}v+\xi(u)\dashv_{B}\xi(v))\in \Gamma_\xi $$
 $$(u+\xi(u))\vdash(v+\xi(v))=(u\vdash_{A}v+\xi(u)\vdash_{B}\xi(v))\in \Gamma_\xi.$$
Furthermore, $(\alpha_A\oplus\alpha_B)(\Gamma_\xi)\subset \Gamma_\xi,\, (\beta_A\oplus\beta_B)(\Gamma_\xi)\subset \Gamma_\xi,$ implies 
$$
(\alpha_A\oplus\alpha_B)(u, \xi(u))=(\alpha_A(u),\alpha_B\circ\xi(u))\in \Gamma_\xi,\,(\beta_A\oplus\beta_B)(u, \xi(u))
=(\beta_A(u),\beta_B\circ\xi(u))\in \Gamma_\xi,  
$$
which is equivalent to the condition $\alpha_B\circ\xi(u)=\xi\circ\alpha_A(u),$ i.e $\alpha_B\circ\xi=\xi\circ\alpha_A.$ Similary,
 $\beta_B\circ\xi=\xi\circ\beta_A$. Therefore, $\xi$ is a 
morphism BiHom-associative dialgebras.
\end{proof}

\begin{proposition}\label{P3}
Let $(D, \dashv, \vdash, \alpha, \beta)$ be a BiHom-associative dialgebra and $I$ be a two sided  BiHom-ideal of 
$(D,\dashv, \vdash, \alpha, \beta)$. Then, $(D/I, \left[\cdot, \cdot\right],\overline{\dashv}, \overline{\vdash}, 
\overline{\alpha}, \overline{\beta})$ is a 
BiHom-associative dialgebra where 
$$\overline{x}\;\overline{\dashv}\;\overline{y}:=\overline{x\dashv y},\;\;
 \overline{x}\;\overline{\vdash}\;\overline{y}:=\overline{x\vdash y},\;\;
 \overline{\alpha}(\overline{x}):=\overline{\alpha(x)},\;\;
 \overline{\beta}(\overline{x}):=\overline{\beta(x)},$$
 for all $\overline{x}, \overline{y}\in A/I.$
\end{proposition}

\begin{proof}
We only prove left associativity,  the other being proved similarly. For all $\overline{x}, \overline{y}, 
\overline{z}\in D/I$, 
we have 
\begin{eqnarray}
(\overline{x}\overline{\vdash}\overline{y})\overline{\vdash}\overline{\beta}(\overline{z})
-\overline{\alpha}(\overline{x})\overline{\vdash}(\overline{y}\overline{\vdash}\overline{z})
=\overline{(x\vdash y)\vdash\beta(z)-\alpha(x)\vdash(y\vdash z)}
=0.\nonumber
\end{eqnarray}
Then, $(D/I, \overline{\dashv}, \overline{\vdash}, \overline{\alpha}, \overline{\beta})$ is  BiHom-associative dialgebra.
\end{proof}

Now, let us recall the definition of BiHom-Lie algebra.

\begin{definition}\cite{GACF}
$A$ BiHom-Lie algebra $(L, \left[\cdot, \cdot\right], \alpha,\beta)$ is a $4$-tuple in  where L is linear space, $\alpha,
 \beta : A\rightarrow A $,are linear maps and 
$\left[\cdot, \cdot\right] : L\otimes L\rightarrow L$ is a bilinear maps, such that, for all $x, y, z\in L$ : 
\begin{equation}
\alpha\circ\beta=\beta\circ\alpha,
\end{equation}
\begin{equation}
\alpha(\left[x, y\right])=\left[\alpha(x), \alpha(y)\right],\,\text{and},\,\beta(\left[x, y\right])=\left[\beta(x), \beta(y)\right],
\end{equation}
\begin{equation}
\left[\beta(x),\alpha(y)\right])=-\left[\beta(y), \alpha(x)\right],\, (\text{BiHom-skew-symetry}),
\end{equation}
\begin{equation}
\left[\beta^2(x),\left[\beta(y),\alpha(z)\right]\right]+\left[\beta^2(y),\left[\beta(z),\alpha(x)\right]\right]
+\left[\beta^2(z),\left[\beta(x),\alpha(y)\right]\right]=0,
\end{equation}
\begin{center}
 (\text{BiHom-Jacobi identity}).
\end{center}
\end{definition}
The maps $\alpha$ and $\beta$ (in this order) are called the structure maps of L.

 \begin{definition}
A morphism  between two BiHom-Lie algebras 
$f : (L, [-, -], \alpha, \beta)\rightarrow(L', [-, -]', \alpha', \beta')$ is a linear map 
$f : L\rightarrow L'$ such that $\alpha'\circ f=f\circ\alpha,\, \beta'\circ f=f\circ\beta$ and 
$f(\left[x, y\right])=\left[f(x), f(y)\right]'$, for all $x, y \in L.$
 \end{definition}

The following lemma asserts that the commutator of any BiHom-associative algebra gives rise to BiHom-Lie.
\begin{lemma}\cite{GACF}\label{cll}
 Let $(A, \cdot, \alpha, \beta )$ be a regular BiHom-associative algebra. Then
$$L(A)=(A, [-, -], \alpha, \beta)$$
is a regular BiHom-Lie algebra, where
$$[x, y]=x\cdot y-\alpha^{-1}\beta(y)\cdot\alpha\beta^{-1}(x),$$
for any $x, y\in A.$
\end{lemma}

\begin{proposition}
 Let $(L, [-, -], \alpha,\beta)$ be a BiHom-Lie algebra and $N : L\rightarrow L$ be a Nijenhuis operator on $L$ i.e. 
$\alpha\circ N=N\circ\alpha$, $\beta\circ N=N\circ\beta$ and
\begin{eqnarray}
 [N(x), N(y)]=N([N(x), y]+[x, N(y)]-N([x, y]))\nonumber
\end{eqnarray}
for any $x, y\in L$.
Then, $(L, [-, -]_N, \alpha,\beta)$ is a BiHom-Lie algebra with
\begin{eqnarray}
 [x, y]_N=[N(x), y]+[x, N(y)]-N([x, y])\nonumber
\end{eqnarray}
for all $x, y\in L$.
\end{proposition}
\begin{proof}
 It follows from direct computation.
\end{proof}

\begin{corollary}\label{aln1}
 Let $(A, \cdot, \alpha,\beta)$ be a BiHom-associative algebra  and $N : A\rightarrow A$ be a Nijenhuis operator on $A$ i.e. 
$\alpha\circ N=N\circ\alpha$, $\beta\circ N=N\circ\beta$ and
\begin{eqnarray}
 N(x)\cdot N(y)=N(N(x)\cdot y+x\cdot N(y)-N(x\cdot y))\nonumber
\end{eqnarray}
for any $x, y\in A$.
Let us denote by $L(A)$ the BiHom-Lie algebra associated with $A$ as in Proposition \ref{cll}. Then,
$(A, [-, -]_N, \alpha,\beta)$ is a BiHom-Lie algebra.
\end{corollary}
% \begin{proof}
% We only need to show that $(A, \alpha, \beta, \ast_N)$ is a BiHom-associative algebra. And the Theorem \ref{} will end the proof.
% \end{proof}

\begin{corollary}\label{aln2}
 Let $(A, \cdot, \alpha,\beta)$ be a BiHom-associative algebra and $N : A\rightarrow A$ be a Nijenhuis operator on $A$ i.e. 
$\alpha\circ N=N\circ\alpha$, $\beta\circ N=N\circ\beta$ and
\begin{eqnarray}
 N(x)\cdot N(y)=N(N(x)\cdot y+x\cdot N(y)-N(x\cdot y))\nonumber
\end{eqnarray}
for any $x, y\in A$.
Then, $(A, \{-, -\}, \alpha,\beta)$ is a BiHom-Lie algebra with
 \begin{eqnarray}
 \{x, y\}= x\ast_N y-\alpha^{-1}\beta(y)\ast_N\alpha\beta^{-1}(x)\nonumber
\end{eqnarray}
and 
\begin{eqnarray}
 x\ast_Ny= N(x)\cdot y+x\cdot N(y)-N(x\cdot y)\nonumber
\end{eqnarray}
for all $x, y\in A$.
\end{corollary}
\begin{proof}
% We only need to show that $(A, \alpha, \beta, \ast_N)$ is a BiHom-associative algebra.For all $x, y, z\in A$, one has
% \begin{eqnarray}
%  fgg
% \end{eqnarray}
It is similar to the one of Proposition  \ref{tns}. And the Lemma \ref{cll} will end the proof.
\end{proof}

\begin{remark}
 The BiHom-Lie algebra generated by Corollary \ref{aln1} and Corollary \ref{aln2} are equal.
\end{remark}

\begin{proposition}
 Let $(D, \dashv, \vdash, \alpha, \beta)$ be a BiHom-associative dialgebra. Then,for all $x, y\in D$, the bracket
$$[x, y]=[x, y]_L+[x, y]_R,$$
where
\begin{eqnarray}
 [x, y]_L&=&x\dashv y-\alpha^{-1}\beta(y)\dashv\alpha\beta^{-1}(x),\nonumber\\
{[x, y]_R}&=&x\vdash y-\alpha^{-1}\beta(y)\vdash\alpha\beta^{-1}(x),\nonumber
\end{eqnarray}
is a BiHom-Lie bracket if and only if 
\begin{eqnarray}
\alpha(x)\dashv(y\vdash z)=(x\dashv y)\vdash\beta(z),\label{dil1}\\
\alpha(x)\dashv(y\dashv z)=(x\vdash y)\vdash\beta(z).\label{dil2}
\end{eqnarray}
\end{proposition}
\begin{proof} It is essentialy based on Lemma \ref{cll}. That is,
 an expansion of BiHom-Jacobi identity leads to $48$ terms including $8$ terms which cancel pairewise by axiom (\ref{eq4}), $4$ terms
cancel pairewise by axiom (\ref{eq5}), $12$ terms cancel pairewise by axiom (\ref{eq6}), $6$ terms cancel pairewise by axiom (\ref{eq7}) and 
$6$ terms cancel pairewise by axiom (\ref{eq8}).\\
For the of the $12$ terms, $8$ terms cancel pairewise by axiom (\ref{dil1}) and $4$ terms cancel pairewise by axiom (\ref{dil2}).
\end{proof}

\begin{definition}\label{Leib}
A (right ) BiHom-Leibniz algebra is a $4$-tuple $(L, \left[\cdot, \cdot\right], \alpha, \beta)$, where L is a linear space, $\left[\cdot, \cdot\right] : L \times L\rightarrow L$
is a bilinear map and $\alpha, \beta : L\rightarrow L$ are linear maps satisfying 
\begin{equation}
\left[\left[x, y\right], \alpha\beta(z),\right]=\left[\left[x, \beta(z)\right], \alpha(y)\right]+\left[\alpha(x),\left[y, \alpha(z)\right]\right], 
\end{equation}
for all $x, y, z\in L$. 
\end{definition}

\begin{example}
 Let $L$ be a two-dimensional vector space and $\left\{e_1, e_2\right\}$ be a basis of $L$. Then, $(L, [-, -], \alpha, \beta)$ is a
BiHom-Leibniz algebra with
$$\left[e_1, e_2\right]=ae_1,
\left[e_2, e_2\right]=be_1,\;
\alpha(e_2)=\beta(e_2)=e_1, 
 a, b\in \mathbb{R}.$$
\end{example}

Now, we introduce BiHom-Poisson dialgebras and study its connection with BiHom-associative dialgebras.

\begin{definition}\label{pois}
A BiHom-Poisson dialgebra is a BiHom-associative dialgebra $({P}, \dashv, \vdash,  \alpha, \beta )$ and a BiHom-Leibniz algebra
$({P}, [-, -], \alpha, \beta )$ such that
\begin{eqnarray}
{[x\dashv y, \alpha\beta(z)]}&=&\alpha(x)\dashv[y, \alpha(z)]+[x, \beta(z)]\dashv\alpha(y),\nonumber\\
{[x\vdash y, \alpha\beta(z)]}&=&\alpha(x)\vdash[y, \alpha(z)]+[x, \beta(z)]\vdash\alpha(y),\nonumber\\
\{\alpha\beta(x), y\dashv z\}&=&\beta(y)\vdash[\alpha(x), z]+[\beta(x), y]\dashv\beta(z)=[\alpha\beta(x), y\vdash z],\nonumber
\end{eqnarray}
are satisfied for $x, y, z\in{P}$.
\end{definition}

\begin{theorem}
Let $({D}, \dashv, \vdash, \alpha, \beta)$ be a BiHom-associative dialgebra. Then, 
$$P(D)=(D,[-, -],\dashv, \vdash, \alpha, \beta)$$
 is a BiHom-Poisson dialgebra, where  $[x, y]=x \dashv y-y\vdash x$,  for any  $x, y\in {D}$.
\end{theorem}
\begin{proof}
By Theorem \ref{dilei} $P(D)$ is a BiHom-Leibniz algebra. Moreover, for any $x, y, z\in D$,
\begin{eqnarray}
&& [x\dashv y, \alpha\beta(z)]-\alpha(x)\dashv[y, \alpha(z)]-[x, \beta(z)]\dashv\alpha(y)=\nonumber\\
&&=(x\dashv y)\dashv\alpha\beta(z)-\alpha^{-1}\beta\alpha\beta(z)\vdash\alpha\beta^{-1}(x\dashv y)
-\alpha(x)\dashv(y\dashv\alpha(z)-\alpha^{-1}\beta\alpha(z)\vdash\alpha\beta^{-1}(y))\nonumber\\
&&-(x\dashv\beta(z)-\alpha^{-1}\beta\beta(z)\vdash\alpha\beta^{-1}(x))\dashv\alpha(y)\nonumber\\
&&=(x\dashv y)\dashv\alpha\beta(z)-\beta^2(z)\vdash(\alpha\beta^{-1}(x)\dashv \alpha\beta^{-1}(y))-\alpha(x)\dashv(y\dashv\alpha(z)\nonumber\\
&&\quad+\alpha(x)\dashv(\beta(z)\vdash\alpha\beta^{-1}(y))-
(x\dashv\beta(z))\dashv\alpha(y)+(\alpha^{-1}\beta^2(z)\vdash\alpha\beta^{-1}(x))\dashv\alpha(y)\nonumber.
\end{eqnarray}
The last three axioms are proved analagously.  This completes the proof.
\end{proof}

\begin{theorem}
 Let $(P, \dashv, \vdash, [-, -], \alpha, \beta )$ be a  BiHom-Poisson dialgebra and $\alpha', \beta' : D\rightarrow D$ two morphisms of 
BiHom-Poisson dialgebras such that the maps $\alpha, \alpha', \beta, \beta'$ commute pairewise. Then
$$P_{(\alpha', \beta')}=(D, \, \triangleleft:=\dashv(\alpha'\otimes\beta'),\; \triangleright:=\vdash(\alpha'\otimes\beta'),\; 
\{-, -\}:=[-, -](\alpha'\otimes\beta'),\; \alpha\alpha',\; \beta\beta'),$$
is a  BiHom-Poisson dialgebra. 
\end{theorem}
\begin{proof}
 It is essentialy based on the one of Theorem \ref{tw}.
\end{proof}

Now, we introduce action of BiHom-Leibniz algebra on another one.

\begin{definition}
Let $D$ and $L$ be two BiHom-Leibniz algebras. An action of $D$ on $L$ consists of a pair of bilinear maps, 
$D\times L\rightarrow L, (x, a)\mapsto [x, a]$ and $L\times D\rightarrow [x, a]$, such that
\begin{eqnarray}\label{le}
\left[\alpha(x),\left[a, \alpha(b)\right]\right]&=&\left[\left[x, a\right], \alpha\beta(b),\right]
-\left[\left[x, \beta(b)\right], \alpha(a)\right]\label{la1}\\
\left[\alpha(a),\left[x, \alpha(b)\right]\right]&=&\left[\left[a, x\right], \alpha\beta(b),\right]
-\left[\left[a, \beta(b)\right], \alpha(x)\right]\\
\left[\alpha(a),\left[b, \alpha(x)\right]\right]&=&\left[\left[a, b\right], \alpha\beta(x),\right]
-\left[\left[a, \beta(x)\right], \alpha(b)\right]\\
\left[\alpha(a),\left[x, \alpha(y)\right]\right]&=&\left[\left[a, x\right], \alpha\beta(y),\right]
-\left[\left[a, \beta(y)\right], \alpha(x)\right]\\
\left[\alpha(x),\left[a, \alpha(y)\right]\right]&=&\left[\left[x, a\right], \alpha\beta(y),\right]
-\left[\left[x, \beta(y)\right], \alpha(a)\right]\\
\left[\alpha(x),\left[y, \alpha(a)\right]\right]&=&\left[\left[x, y\right], \alpha\beta(a),\right]
-\left[\left[x, \beta(a)\right], \alpha(y)\right]
\end{eqnarray}
for all $x, y\in D, a, b\in L$.
\end{definition}

\begin{lemma}
Given a BiHom-Leibniz action of $D$ on $L$, we can consider the  {\it semidirect product} Leibniz algebra $L\rJoin D$, which consists of
vector space $D\oplus L$ together with the Leibniz bracket given by
\begin{eqnarray}
 [(x, a), (y, b)]=([x, y]+[x, b]+[a, y], [a, b])
\end{eqnarray}
for all $(x, a), (x, b)\in D\times L$.
\end{lemma}
\begin{proof}
 \begin{eqnarray}
  [\alpha(x, a), [(y, b), \alpha(z, c)]]
&=&[(\alpha(x), \alpha(a)), ([y, \alpha(z)]+[y, \alpha(c)]+[b, \alpha(z)], [b, \alpha(c)])]\nonumber\\
&=&\Big([\alpha(x), [y, \alpha(z)]]+[\alpha(x), [y, \alpha(c)]]+[\alpha(x), [b, \alpha(z)]]+[\alpha(x), [b, \alpha(c)]]\nonumber\\
&&+[\alpha(a), [y, \alpha(z)]]+[\alpha(a), [y, \alpha(z)]]+[\alpha(a), [y, \alpha(c)]]
+[\alpha(a), [b, \alpha(z)],\nonumber\\
&& [\alpha(a), [b, \alpha(c)]]\Big).\nonumber\\
%  \end{eqnarray}
%  \begin{eqnarray}
  {[[(x, a), (y, b)], \alpha\beta(z, c)]}&=&[([x, y]+[x, b]+[a, y]), [a, b]), (\alpha\beta(z), \alpha\beta(c))]\nonumber\\
&=&([[x, y], \alpha\beta(z)]+[[x, b], \alpha\beta(z)]+[[a, y], \alpha\beta(z)]+[[x, y], \alpha\beta(c)]\nonumber\\
&&+[[x, b], \alpha\beta(c)]
+[[a, y], \alpha\beta(c)]+[[a, b], \alpha\beta(c)], [[a, b], \alpha\beta(c)].\nonumber\\
%  \end{eqnarray}
% \begin{eqnarray}
 {[[(x, a), \beta(z, c)], \alpha(y, b)]}
&=&[([x, \beta(z)]+[x, \beta(c)]+[a, \beta(z)], [a, \beta(c)]), (\alpha(y), \alpha(b))]\nonumber\\
&=&([[x, \beta(z)], \alpha(y)]+[[x, \beta(c)], \alpha(y)]+[[a, \beta(z)], \alpha(y)]+[[x, \beta(z)], \alpha(b)]\nonumber\\
&&+[[x, \beta(c)], \alpha(b)]+[[a, \beta(z)], \alpha(b)]+[[a, \beta(c)], \alpha(y)], [[a, \beta(c)], \alpha(b)])\nonumber.
\end{eqnarray}
Using axioms in Definition \ref{le}, it follows that 
$${[[(x, a), (y, b)], \alpha\beta(z, c)]}={[[(x, a), \beta(z, c)], \alpha(y, b)]}+[\alpha(x, a), [(y, b), \alpha(z, c)]].$$
Which proves the proposition.
\end{proof}

\begin{theorem}\label{dilei}
Let $({D}, \dashv, \vdash, \alpha, \beta)$ be a regular BiHom-associative dialgebra. Then the bracket defined by
 $\left[x, y\right]=x \dashv y-\alpha^{-1}\beta(y)\vdash\alpha\beta^{-1}(x)$, defines a structure of  BiHom-Leibniz algebra on ${D}$, and denoted
${\bf Lb}(D)$.
\end{theorem}
\begin{proof}
% The proof following by a straightforward computation in which the definitions \ref{Leib} and \ref{dia} are used once. In fact, 
For any $x, y, z\in {D}$, we have 
\begin{eqnarray}
[[x, y], \alpha\beta(z)]
&=&(x\dashv y-\alpha^{-1}\beta(y)\vdash\alpha\beta^{-1}(x))\dashv\alpha\beta(z)\nonumber\\
&&-\alpha^{-1}\beta\alpha\beta(z)\vdash(x\dashv y-\alpha^{-1}\beta(y)\vdash\alpha\beta^{-1})\nonumber\\
&=&(x\dashv y)\dashv\alpha\beta(z)-(\alpha^{-1}\beta(y)\vdash\alpha\beta^{-1}(x))\dashv\alpha\beta(z)\nonumber\\
&&-\beta^2(z)\vdash(\alpha\beta^{-1}(x)\dashv\alpha\beta^{-1}(y)
+\beta^2(z)\vdash(y\vdash\alpha^2\beta^{-2}(x)).\nonumber\\
% \end{eqnarray}
% \begin{eqnarray}
 {[[x, \beta(z)], \alpha(y)]}
&=&(x\dashv\beta(z)-\alpha^{-1}\beta^2(z)\vdash\alpha\beta^{-1}(x)\dashv\alpha(y)\nonumber\\
&&-\alpha^{-1}\beta\alpha(z)\vdash(\alpha\beta^{-1}(x\dashv y+\alpha^{-1}\beta(y)\vdash\alpha\beta^{-1})\nonumber\\
&=&(x\dashv\beta(z))\dashv\alpha(y)-(\alpha^{-1}\beta^2(z)\vdash\alpha\beta^{-1}(x))\dashv\alpha(y)\nonumber\\
&&-\beta(y)\vdash(\alpha\beta^{-1}(x)\dashv\alpha(z)) +\beta(y)\vdash(\beta(z)\vdash\alpha^2\beta^{-2}(x))\nonumber.\\
% \end{eqnarray}
% \begin{eqnarray}
 {[\alpha(x), [y, \alpha(z)]]}
&=&\alpha(x)\dashv(y\dashv\alpha(z)-\alpha^{-1}\beta\alpha(z)\vdash\alpha\beta^{-1}(y))
-\alpha^{-1}\beta(y\dashv\alpha(z)\nonumber\\
&&-\beta(z)\vdash\alpha\beta^{-1}(y))\vdash\alpha\beta^{-1}\alpha(x)\nonumber\\
&=&\alpha(x)\dashv(y\dashv\alpha(z))-\alpha(x)\dashv(\beta(z)\vdash\alpha\beta^{-1}(y))\nonumber\\
&&-(\alpha^{-1}\beta(y)\dashv\beta(z))\vdash\alpha^2\beta^{-1}(x) -(\alpha^{-1}\beta^2(z)\vdash y)\vdash\alpha^2\beta^{-1}(x)\nonumber.
\end{eqnarray}
By axioms in Definition \ref{dia}, the conclusion holds.
\end{proof}

In the relations contained in the below definition, we omitted the subsript for simplifying the typography.
\begin{definition}\label{act1}
Let $D$ and $L$ be dialgebras. An action of $D$ on $L$ consists of four linear maps, two of them denoted by the symbol $\dashv$ and other two by 
$\vdash$,
\begin{eqnarray}
 \dashv : D\otimes L\rightarrow L, & & \dashv : L\otimes D\rightarrow L,\nonumber\\
\dashv : D\otimes L\rightarrow L, & & \dashv : L\otimes D\rightarrow L\nonumber
\end{eqnarray}
such that the following $30$ equalities hold :\\

\begin{tabular}{lr}
${(01)}\quad (x\dashv a)\dashv\beta(b)=\alpha(x)\dashv(a\dashv b) $,
&$\qquad\qquad{(16)}\quad (a\dashv x)\dashv\beta(y)=\alpha(a)\dashv(x\dashv y)$,\\
${(02)}\quad (x\dashv a)\dashv\beta(b)=\alpha(x)\dashv(a\vdash b)$,
&$\qquad\qquad {(17)}\quad (a\dashv x)\dashv\beta(y)=\alpha(a)\dashv(x\vdash y)$,\\
${(03)}\quad (x\vdash a)\dashv\beta(b)=\alpha(x)\vdash(a\dashv b)$,
&$\qquad\qquad {(18)}\quad (a\vdash x)\dashv\beta(y)=\alpha(a)\vdash(x\dashv y)$,\\
${(04)}\quad (x\dashv a)\vdash\beta(b)=\alpha(x)\vdash(a\vdash b)$,
&$\qquad\qquad {(19)}\quad (a\dashv x)\vdash\beta(y)=\alpha(a)\vdash(x\vdash y)$,\\
${(05)}\quad (x\vdash a)\vdash\beta(b)=\alpha(x)\vdash(a\vdash b)$,
&$\qquad\qquad {(20)}\quad (a\vdash x)\vdash\beta(y)=\alpha(a)\vdash(x\vdash y)$,\\
\\
% \end{tabular} {}
% \begin{tabular}{lr}
${(06)}\quad (a\dashv x)\dashv\beta(b)=\alpha(a)\dashv(x\dashv b)$,
&$\qquad\qquad {(21)}\quad (x\dashv a)\dashv\beta(y)=\alpha(x)\dashv(a\dashv y)$,\\
${(07)}\quad (a\dashv x)\dashv\beta(b)=\alpha(a)\dashv(x\vdash b)$,
&$\qquad\qquad{(22)}\quad  (x\dashv a)\dashv\beta(y)=\alpha(x)\dashv(a\vdash y)$,\\
${(08)}\quad (a\vdash x)\dashv\beta(b)=\alpha(a)\vdash(x\dashv b)$,
&$\qquad\qquad  {(23)}\quad (x\vdash a)\dashv\beta(y)=\alpha(x)\vdash(a\dashv y)$,\\
${(09)}\quad (a\dashv x)\vdash\beta(b)=\alpha(a)\vdash(x\vdash b)$,
&$\qquad\qquad{(24)}\quad (x\dashv a)\vdash\beta(y)=\alpha(x)\vdash(a\vdash y)$,\\
${(10)}\quad (a\vdash x)\vdash\beta(b)=\alpha(a)\vdash(x\vdash b)$,
&$\qquad\qquad {(25)}\quad (x\vdash a)\vdash\beta(y)=\alpha(x)\vdash(a\vdash y)$,\\
\\
\end{tabular}

\begin{tabular}{lr}
${(11)}\quad (a\dashv b)\dashv\beta(x)=\alpha(a)\dashv(b\dashv x)$,
&$\qquad\qquad {(26)}\quad (x\dashv y)\dashv\beta(a)=\alpha(x)\dashv(y\dashv a)$,\\
${(12)}\quad (a\dashv b)\dashv\beta(x)=\alpha(a)\dashv(b\vdash x)$,
&$\qquad\qquad  {(27)}\quad (x\dashv y)\dashv\beta(a)=\alpha(x)\dashv(y\vdash a)$,\\
${(13)}\quad (a\vdash b)\dashv\beta(x)=\alpha(a)\vdash(b\dashv x)$,
&$\qquad\qquad {(28)}\quad (x\vdash y)\dashv\beta(a)=\alpha(x)\vdash(y\dashv a)$,\\
${(14)}\quad (a\dashv b)\vdash\beta(x)=\alpha(a)\vdash(b\vdash x)$,
&$\qquad\qquad{(29)}\quad  (x\dashv y)\vdash\beta(a)=\alpha(x)\vdash(y\vdash a)$,\\
${(15)}\quad (a\vdash b)\vdash\beta(x)=\alpha(a)\vdash(b\vdash x)$,
&$\qquad\qquad{(30)}\quad  (x\vdash y)\vdash\beta(a)=\alpha(x)\vdash(y\vdash a)$,
\end{tabular}
\\
\\
for all $x, y\in D, a, b\in L$. The action is called trivial if these four maps are trivial.
\end{definition}

\begin{example}
i) Any BiHom-associative dialgebra may be seen as acting on itself
 ii)Given a homomorphism $\varphi : D\rightarrow L$ of BiHom-associative dialgebras, then there is an action of $D$ on $L$ via the maps
$x\triangleleft a:=\varphi(x)\dashv a,\; x\triangleright a:=\varphi(x)\triangleright, a\; 
a\triangleleft x:=a\vdash\varphi(x)\;\;\mbox{and}\;\; a\triangleright x:=a\vdash\varphi(x)$.\\
iii)If $\psi : L\rightarrow D$ is an isomorphism of BiHom-associative dialgebras, then there is an action of $D$ on $L$ via the maps
$x\triangleleft a:=\psi^{-1}(x)\dashv a,\; x\triangleright a:=\psi^{-1}(x)\triangleright a,\; 
a\triangleleft x:=a\vdash\psi^{-1}(x)\;\;\mbox{and}\;\; a\triangleright x:=a\vdash\psi^{-1}(x)$.\\
iv) If $I$ is an ideal of $D$, then the left and the right product yield an action of $D$ on I. 
\end{example}
% \begin{remark}
% When $L$ is an abelian dialgebra, the action of $D$ on $L$ is nothing but the bimodule structure of $D$ on $L$.
% Thus, if $L$ is a bimodule over a BiHom-associative dialgebra $D$ (\cite{}, Definition ...), thought as an abelian dialgebra, 
% then the bimodule structure defines an action of $D$ on the abelian dialgebra $L$.
% \end{remark}

\begin{lemma}\label{lbs}
Given two regular BiHom-associative dialgebras $D$ and $L$ together with an action of $D$ on $L$, there is an action an action
${\bf Lb} (D)$ on ${\bf Lb} (L)$ given by
\begin{eqnarray}
 [x, a]&=&x\dashv a-\alpha^{-1}\beta(a)\dashv\alpha\beta^{-1}(x),\nonumber\\
{[a, x]}&=&a\vdash x-\alpha^{-1}\beta(x)\vdash\alpha\beta^{-1}(a),\nonumber
\end{eqnarray}
for all $x\in {\bf Lb} (D)$, $a\in {\bf Lb} (L)$.
\end{lemma}
\begin{proof}
For all $x\in {\bf Lb} (D)$, $a\in {\bf Lb} (L)$,
\begin{eqnarray}
[[x, a], \alpha\beta(b)]
&=&(x\dashv a-\alpha^{-1}\beta(a)\vdash\alpha\beta^{-1}(x))\dashv\alpha\beta(b)\nonumber\\
&&-\alpha^{-1}\beta\alpha\beta(b)\vdash\alpha\beta^{-1}(x\dashv a-\alpha^{-1}\beta(a)\vdash\alpha\beta^{-1}(x))\nonumber\\
&=&(x\dashv a)\dashv\beta\alpha(b)-(\alpha^{-1}\beta(a)\vdash\alpha\beta^{-1}(x))\dashv\beta\alpha(b)\nonumber\\
&&-\beta^2(b)\vdash(\alpha\beta^{-1}(x)\dashv\beta^{-1}\alpha(a))
+\beta^2(b)\vdash(a\vdash\alpha^2\beta^{-2}(x))\nonumber.
\end{eqnarray}

On the other hand,
\begin{eqnarray}
&&\qquad [[x, \beta(b)], \alpha(a)]+[\alpha(x), [a, \alpha(b)]]=\nonumber\\
&&=\Big(x\dashv\beta(b)-\alpha^{-1}\beta^2(b)\vdash\alpha\beta^{-1}(x)\Big)\dashv\alpha(a)
-\alpha^{-1}\beta\alpha(a)\vdash\alpha\beta^{-1}\Big(x\dashv\beta(b)-\alpha^{-1}\beta^2(b)\vdash\alpha\beta^{-1}(x)\Big)\nonumber\\
&&\quad+\alpha(x)\dashv\Big(a\dashv \alpha(b)-\alpha^{-1}\beta\alpha(b)\vdash\alpha\beta^{-1}(a)\Big)
-\alpha^{-1}\beta\Big(a\dashv \alpha(b)-\alpha^{-1}\beta\alpha(b)\vdash\alpha\beta^{-1}(a)\Big)\vdash\alpha\beta^{-1}\alpha(x)\nonumber\\
&&=(x\dashv\beta(b))\dashv\alpha(a)-(\alpha^{-1}\beta^2(b)\vdash\alpha\beta^{-1}(x))\dashv\alpha(a)
-\beta(a)\vdash (\alpha\beta^{-1}(x)\dashv\alpha(b))\nonumber\\
 &&\quad+\beta(a)\vdash(\beta(b)\vdash\alpha^2\beta^{-2}(x))+\alpha(x)\dashv(a\dashv\alpha(b))-\alpha(x)\dashv(\beta(b)\vdash\alpha\beta^{-1}(a))\nonumber\\
&&\quad-(\alpha^{-1}\beta(a)\dashv\beta(b))\vdash\beta^{-1}\alpha^2(x)+(\alpha^{-1}\beta^2(b)\vdash a)\vdash\beta^{-1}\alpha^2(x)\nonumber.
\end{eqnarray}
Using axioms (\ref{eq5}), (\ref{eq7}), it comes
\begin{eqnarray}
&&\qquad [[x, \beta(b)], \alpha(a)]+[\alpha(x), [a, \alpha(b)]]=\nonumber\\
&&=-(\alpha^{-1}\beta^2(b)\vdash\alpha\beta^{-1}(x))\dashv\alpha(a)-\beta(a)\vdash (\alpha\beta^{-1}(x)\dashv\alpha(b))\nonumber\\
&&\quad+\alpha(x)\dashv(a\dashv\alpha(b)) +(\alpha^{-1}\beta^2(b)\vdash a)\vdash\alpha^2\beta^{-1}(x)\nonumber.
\end{eqnarray}
By comparing, we get the attended result. The five other axioms are proved in the same way.
\end{proof}

\begin{lemma}\label{lbsa}
Let $D$ and $L$ be two regular BiHom-associative dialgebras together with an action of $D$ on $L$. There is a BiHom-associative dialgebra structure on 
$L\rJoin D$ which consists with vector space $L\oplus D$ and 
 \begin{eqnarray}
 (a, x)\triangleleft (b, y)&=&(a\dashv b+a\dashv y+x\dashv b, x\dashv y)\nonumber,\\
(a, x)\triangleright (b, y)&=&(a\vdash b+a\vdash y+x\vdash b, x\vdash y)\nonumber,
 \end{eqnarray}
for any $(a, x), (b, y)\in L\times D$
\end{lemma}
\begin{proof}
For any $a, b, c\in L, x, y, z\in D$, one has
 \begin{eqnarray}
&&\qquad  \Big((a, x)\triangleleft(b, y)\Big)\triangleleft\beta(c, z)-\alpha(a, x)\triangleleft\Big((b, y)\triangleright(c, z)\Big)=\nonumber\\
&&=(a\dashv b+a\dashv y+x\dashv b, x\dashv y)\triangleleft(\beta(c), \beta(z))-(\alpha(a), \alpha(x))\triangleleft
(b\vdash c+b\vdash z+y\vdash c, y\vdash z)\nonumber\\
&&=\Big((a\dashv b+a\dashv y+x\dashv b)\dashv\beta(c)+(a\dashv b+a\dashv y+x\dashv b)\dashv\beta(z)
+(x\dashv y)\dashv\beta(c),\;\; (x\dashv y)\dashv\beta(z)\Big)\nonumber\\
&&\quad-\Big(\alpha(a)\dashv(b\vdash c+b\vdash z+y\vdash c)+\alpha(a)\dashv(y\vdash z)
+\alpha(x)\dashv(b\vdash c+b\vdash z+y\vdash c),\;\; \alpha(x)\dashv(y\vdash z)\Big)\nonumber\\
&&=\Big((a\dashv b)\dashv\beta(c)+(a\dashv y)\dashv\beta(c)
+(x\dashv b)\dashv\beta(c)+(a\dashv b)\dashv\beta(z)+(a\dashv y)\dashv\beta(z)+(x\dashv b)\dashv\beta(z)\nonumber\\
&&\quad+(x\dashv y)\dashv\beta(c)-\alpha(a)\dashv(b\vdash c)-\alpha(a)\dashv(b\vdash z)-\alpha(a)\dashv(y\vdash c)-\alpha(a)\dashv(y\vdash z)
-\alpha(x)\dashv(b\vdash c)\nonumber\\
&&\quad-\alpha(x)\dashv(b\vdash z)-\alpha(x)\dashv(y\vdash c),\;\; (x\dashv y)\dashv\beta(z)-\alpha(x)\dashv(y\vdash z)\Big)\nonumber.
 \end{eqnarray}
The left hand side vanishes by axiom (\ref{eq5}) and axioms $(02), (07), (12), (17), (22), (27)$ in Definition \ref{act1}.
The other axioms are proved in the same way.
\end{proof}

\begin{theorem}
 Let $D$ and $L$ be two regular BiHom-associative dialgebras together with an action of $D$ on $L$.
Then, ${\bf Lb}(L\rJoin D)={\bf Lb}(L)\rJoin{\bf Lb}(D)$.
\end{theorem}
\begin{proof}
By lemma $\ref{lbs}$, ${\bf Lb}(D)$ acts on ${\bf Lb}(L)$, so it makes sense to consider the semidirect product Leibniz algebra 
${\bf Lb}(L)\rJoin{\bf Lb}(D)$. It is clear that ${\bf Lb}(L\rJoin D)$ and ${\bf Lb}(L)\rJoin{\bf Lb}(D)$ are egal as vector space, so we only need 
to verify that they share the same bracket. Let $(a, x), (b, y)\in L\times D$. If we use the bracket in ${\bf Lb}(L)\rJoin{\bf Lb}(D)$, we get :
\begin{eqnarray}
 [(a, x), (b, y)]
&=&([a, b]+[x, b]+[a, y], [x, y])\nonumber\\
&=&(a\dashv y-b\vdash+x\dashv b-b\vdash x+a\dashv y-y\vdash a, x\dashv y-y\vdash x).\nonumber
\end{eqnarray}
On the other hand, if we use the Leibniz bracket in ${\bf Lb}(L\rJoin D)$ (Lemma \ref{lbsa}), we get
 \begin{eqnarray}
 \{(a, x), (b, y)\}
&=&(a, x)\triangleleft(b, y)-(b, y)\triangleright (a, x)\nonumber\\
&=&(a\dashv b+x\dashv b+a\dashv y, x\dashv y)-(b\vdash a+y\vdash a+b\vdash x, y\vdash x)\nonumber,
 \end{eqnarray}
So the brackets are equal.
\end{proof}

\section{Central extensions}
This section concerns the central extension of BiHom-associative dialgebras in relation with cocycles.
\begin{definition}
 Let $(D_i, \dashv_i, \vdash_i, \alpha_i, \beta_i), i=1,2,3$ be three BiHom-associative dialgebras. The BiHom-associative dialgebra $D_2$ is called
 the extension of $D_3$ by $D_1$ if there are homomorphisms $\phi : D_1\rightarrow D_2$ and $\psi : D_2\rightarrow D_3$ such that the following sequence
$$0\rightarrow D_1\stackrel{\phi}{\longrightarrow}D_2\stackrel{\psi}{\rightarrow}D_3\rightarrow0$$
is exact. 
\end{definition}

\begin{definition}
 An extension is called trivial if there exists a BiHom-ideal $I$ of $D_2$ complementary to $Ker\psi$ i.e.
$$D_2=Ker\psi\oplus I$$
\end{definition}

It may happen that there exist several extensions of $D_3$ by $D_1$. To classify extensions the notion of equivalent extensions is defined.

\begin{definition}
 Two sequences 
$$0\rightarrow D_1\stackrel{\phi}{\longrightarrow}D_2\stackrel{\psi}{\rightarrow}D_3\rightarrow0$$
and 
$$0\rightarrow D_1\stackrel{\phi'}{\longrightarrow}D_2\stackrel{\psi'}{\rightarrow}D_3\rightarrow0$$
are equivalent extensions if there exists a associative dialgebra isomorphism $f : D_2\rightarrow D'_2$ such that
$$f\circ\phi=\phi'\quad\mbox{and}\quad \psi'\circ f=\psi.$$
\end{definition}

\begin{definition}
 An extension
$$0\rightarrow D_1\stackrel{\phi}{\longrightarrow}D_2\stackrel{\psi}{\rightarrow}D_3\rightarrow0$$
is called central if the kernel of $\psi$ is contained in the center $Z(D_2)$ of $D_2$, i.e. $Ker\psi\subset Z(D)$.
\end{definition}

Now, we introduce $2$-cocycle on BiHom-associative dialgebra with values in a BiHom-module.
\begin{definition}
 Let $(D, \dashv, \vdash, \alpha, \beta)$ be a BiHom-associative dialgebra and $(M, \alpha_M, \beta_M)$ a BiHom-module over the same field that $D$. A pair 
$\Theta=(\theta_1, \theta_2)$ of bilinear maps $\theta_1 : D\times D\rightarrow V$ and $\theta_2 : D\times D\rightarrow V$ is called a $2$-cocycle on
$D$ with values in $V$ if $\theta_1$ and $\theta_2$ satisfy 
\begin{eqnarray}
 \theta_1(x\dashv y, \beta(z))&=&\theta_1(\alpha(x), y\dashv z), \label{cc1}\\
\theta_1(x\dashv y, \beta(z))&=&\theta_1(\alpha(x), y\vdash z), \\
\theta_2(x\vdash y, \beta(z))&=&\theta_2(\alpha(x), y\vdash z), \\
\theta_2(x\dashv y, \beta(z)) &=& \theta_2(\alpha(x), y\vdash z), \\
\theta_1(x\vdash y, \beta(z))&=&\theta_2(\alpha(x), y\dashv z),
\end{eqnarray}
 for all $x, y, z\in D$.
\end{definition}
The set of all $2$-cocycles on $D$ with values in $M$ is denoted $Z^2(D, M)$, which a vector space. \\
In the below lemma, we give a special type of $2$-cocycles which are called $2$-coboundaries.
\begin{lemma}
 Let $\nu : D\rightarrow V$ be a linear map, and define $\varphi_1(x, y)=\nu(x\dashv y)$ and $\varphi_2(x, y)=\nu(x\vdash y)$. Then,
$\Phi=(\varphi_1, \varphi_2)$ is a $2$-cocycle on $D$.
\end{lemma}
\begin{proof}
We will prove one equality, the others being proved in the same way. For any $x, y, z\in D$, one has
 \begin{eqnarray}
  \varphi_1(\alpha(x), y\dashv z)&=&\nu(\alpha(x)\dashv(y\dashv z))=\nu((x\dashv y)\dashv\beta(z))\nonumber\\
&=&\nu(\alpha(x)\dashv(y\vdash z))=\varphi_1(\alpha(x), y\vdash z).\nonumber
 \end{eqnarray}
This finishes the proof.
\end{proof}
The set of all $2$-coboundaries is denoted by $B^2(D, M)$ and it is a subgroup of $Z^2(D, M)$.
The group $H^2(D, M)=Z^2(D, M)/B^2(D, M)$ is said to be a second cohomology group of $D$ with values in $M$.
Two cocycles $\Theta_1$ and $\Theta_2$ are said to be cohomologous cocycles if $\Theta_1-\Theta_2$ is a coboundary.
\begin{theorem}
 Let $(D, \dashv, \vdash, \alpha_D, \beta_D)$ be a BiHom-associative dialgebra, $(M, \alpha_M, \beta_M)$ a BiHom-module,
$$\theta_1 : D\times D\rightarrow M\quad\mbox{and}\quad \theta_2 : D\times D\rightarrow M$$
be bilinear maps. Let us set $D_\Theta=D\oplus M$, where $\Theta=(\theta_1, \theta_2)$. For any $x, y\in D$, $v, w\in M$, let us define
\begin{eqnarray}
 (x+u)\triangleleft(y+v)=x\dashv y+\theta_1(x, y)\quad\mbox{and}\quad (x+u)\triangleright(y+v)=x\vdash y+\theta_2(x, y).\nonumber
\end{eqnarray}
 Then, $(D_\Theta, \triangleleft, \triangleright, \alpha_A\otimes \alpha_M, \beta_A\otimes \beta_M)$ is a BiHom-associative dialgebra if and only 
if $\Theta$ is a $2$-cocycle.
\end{theorem}
\begin{proof}
 For any $x, y, z\in D, u, v, w\in M$, we have
\begin{eqnarray}
 && ((x+v)\triangleleft(y+w))\triangleleft(\beta(z)+w)-(\alpha(x)+v)\triangleleft((y+w)\triangleleft(z+w))=\nonumber\\
&=&((x+v)\triangleleft(y+w))\triangleleft(\beta(z)+w)-(\alpha(x)+ v)\triangleleft((y\dashv z)+\theta_1(y, z))\nonumber\\
&=&((x\dashv y)\dashv \beta(z))+\theta_1(x\dashv y, \beta(z))-(\alpha(x)\dashv (y\dashv z))-\theta_1(\alpha(x), y\dashv z).\nonumber
\end{eqnarray}
The left hand vanishes by axioms (\ref{eq4}) and (\ref{cc1}). The other axioms are proved analagously.
\end{proof}

\begin{lemma}
Let $\Theta$ be a $2$-cocycle and $\Phi$ a $2$-coboundary. Then, $D_{\Theta+\Phi}$ is a BiHom-associative dialgebra with
\begin{eqnarray}
 (x+u)\unlhd(y+v)=x\dashv y+\varphi_1(x, y)+\theta_1(x, y),\nonumber\\ 
(x+u)\unrhd(y+v)=x\vdash y+\varphi_2(x, y)+\theta_2(x, y)\nonumber.
\end{eqnarray}
Moreover, $D_\Theta\cong D_{\Theta+\Phi}$.
\end{lemma} 
\begin{proof}
First, we have to shown that $D_{\Theta+\Phi}$ is a BiHom-associative dialgebra. So, for any $x+u,\\ y+v, z+w\in D\oplus M$,
\begin{eqnarray}
 &&\qquad((x+u)\unlhd(y+v))\unlhd\beta(z+w)-\alpha(x+u)\unlhd((y+v))\unlhd(z+w))=\nonumber\\
&&=(x\dashv y+\varphi_1(x, y)+\theta_1(x, y))\unlhd(\beta(z)+\beta(w))
-(\alpha(x)+\alpha(u))\unlhd(y\dashv z+\varphi_1(y, z)+\theta_1(y, z))\nonumber\\
&&=(x\dashv y)\dashv\beta(z)+\varphi_1(x\dashv y,\beta(z))+\theta_1(x\dashv y,\beta(z))
-\alpha(x)\dashv y\dashv z-\varphi_1(\alpha(x), y\dashv z)+\theta_1(\alpha(x), y\dashv z)\nonumber
\end{eqnarray}
The left hand side vanishes by (\ref{eq4}) and (\ref{cc1}). The proofs of the rest of axioms are leaved to the reader.\\
Next, the isomorphism $f : D_\Theta\rightarrow D_{\Theta+\Phi}$ is given by 
$f(x+v)=x+\nu(x)+v$. In fact, it is clear that $f$ is a bijective linear map and
 \begin{eqnarray}
  f(\alpha_D+\alpha_M)(x+v)&=&f(\alpha_D(x)+\alpha_M(v))\nonumber\\
&=&\alpha_D(x)+\nu\alpha_D(x)+\alpha_M(v)\nonumber\\
&=&\alpha_D(x)+\alpha_M\nu(x)+\alpha_M(v)\nonumber\\
&=&(\alpha_D+\alpha_M)(x+\nu(x)+v)\nonumber\\
&=&(\alpha_D+\alpha_M)\circ f(x+v).\nonumber
 \end{eqnarray}
Thus, $f$ commutes $\alpha_D+\alpha_M$, and similarly with $\beta_D+\beta_M$.\\
Then, 
\begin{eqnarray}
 f((x+v)\triangleleft (y+w))
&=&f(x\dashv y+\theta_1(x, y))\nonumber\\
&=&f(x\dashv y)+f(\theta_1(x, y))\nonumber\\
&=&x\dashv y+\nu(x\dashv y)+\theta_1(x, y)\nonumber\\
&=&x\dashv y+\varphi_1(x, y)+\theta_1(x, y)\nonumber.
\end{eqnarray}
and
\begin{eqnarray}
 f(x+v)\unlhd f(y+w)
&=&(x+\nu(x)+v)\unlhd(y+\nu(y)+w)\nonumber\\
&=&(x\dashv y)+\varphi_1(x, y)+\theta_1(x, y).\nonumber
\end{eqnarray}
\end{proof}
\begin{corollary}
 Let $\Theta_1, \Theta_2$ be two cohomologous $2$-cocycles on a BiHom-associative dialgebra $D$, and $D_1, D_2$ be the central extensions constructed
with these $2$-cocycles, respectively. The the central extensions $D_1$ and $D_2$ are equivalent extensions. In particular a central extension
defined by a coboundary is equivalent with a trivial central extension.
\end{corollary}

The following theorem is proved Mutatis Mutandis as (\cite{ISR}, Theorem 4.1). So we omitted the proof.

\begin{theorem}
 There exists one to one correspondence between elements of $H^2(D, M)$ and nonequivalents central extensions of associative dialgebra $D$ by $M$.
\end{theorem}

% $L\rtimes D$   $L\rJoin D \circlearrowleft$

% \newpage

\section{Classification }
In this section, we give classification of BiHom-associative dialgebras in low dimension.
% In what follows, we express by mean of vector basis the conditions defining a finite dimensional BiHom-associative dialgebra $(D, \dashv, \vdash
% , \alpha, \beta)$ a
% Let $\{e_i\}, i=1,\dots, n$ be basis of $D$. For any $i, j\in \mathbb{N}, 1\leq i, j\leq n$, let us put 
% $$e_i\dashv e_j=\sum_{k=1}^{n}\gamma_{ij}^ke_k,\quad e_i\vdash e_j\sum_{k=1}^{n}\delta_{ij}^ke_k,\quad\alpha(e_j)=\sum_{k=1}^{n}\alpha_{kj}e_k,\quad
% \beta(e_j)=\sum_{k=1}^{n}\beta_{kj}e_k.$$

% Then we have the below lemma : 
% \begin{lemma}
Let $(D, \dashv, \vdash, \alpha, \beta)$ be an $n$-dimensional BiHom-associative dialgebra, $\{e_i\}$ be a basis of $D$. For any $i, j\in \mathbb{N}, 1\leq i, j\leq n$, let us put 
$$e_i\dashv e_j=\sum_{k=1}^{n}\gamma_{ij}^ke_k,\quad e_i\vdash e_j\sum_{k=1}^{n}\delta_{ij}^ke_k,\quad\alpha(e_j)=\sum_{k=1}^{n}\alpha_{kj}e_k,\quad
\beta(e_j)=\sum_{k=1}^{n}\beta_{kj}e_k.$$
The axioms in Definition \ref{dia} are respectively equivalent to
\begin{eqnarray}
\beta_{kj}\alpha_{pk}-\alpha_{ji}\beta_{pj}&=&0,\\
\gamma_{ij}^p\beta_{qk}\gamma_{pq}^r-\alpha_{pi}\gamma_{jk}^q\gamma_{pq}^r&=&0,\\
\gamma_{ij}^p\beta_{qk}\gamma_{pq}^r-\alpha_{pi}\delta_{jk}^q\gamma_{pq}^r&=&0,\\
\gamma_{ij}^p\beta_{qk}\delta_{pq}^r-\alpha_{pi}\delta_{jk}^q\delta_{pq}^r&=&0,\\
\delta_{ij}^p\beta_{qk}\gamma_{pq}^r-\alpha_{pi}\gamma_{jk}^q\delta_{pq}^r&=&0,\\
\delta_{ij}^p\beta_{qk}\gamma_{pq}^r-\alpha_{pi}\delta_{jk}^q\delta_{pq}^r&=&0.
\end{eqnarray}
% \end{lemma}
\subsection{One dimensional}
There is only one $1$-dimensional BiHom-associative dialgebra ; the nul (or trivial) BiHom-associative dialgebra.

\subsection{Two dimensional}

\begin{tabular}{||c||c||c||c||c||c||c||}
\hline
$Algebras$&Multiplications &Morphisms $\alpha,\beta$.
\\ \hline 
$\mathcal{A}lg_1$&
$\begin{array}{ll}  
e_1\dashv e_2=ae_1,\\
e_2\dashv e_1=be_1,\\
e_1\vdash e_2=ce_1,\\
e_2\vdash e_1=de_1,\\
e_2\vdash e_2=fe_1.
\end{array}$
&
$\begin{array}{ll}  
\alpha(e_2)=e_1,\\
\beta(e_2)=e_1
\end{array}$
\\ \hline 
$\mathcal{A}lg_2$
&
$\begin{array}{ll}  
e_1\dashv e_2=ae_1,\\
e_2\dashv e_1=ae_1,\\
e_2\dashv e_2=e_1,\\
e_1\vdash e_2=e_1,\\
e_2\vdash e_1=e_1,
\end{array}$
&
$\begin{array}{ll}  
\alpha(e_2)=e_1,\\
\beta(e_2)=e_1
\end{array}$
\\ \hline
% \end{tabular}
% 
% 
% \begin{tabular}{||c|c|c|c|c|c|c||}
% \hline
% $Algebras$&Multiplications&Morphisms $\alpha,\beta$.
% \\ \hline
$\mathcal{A}lg_3$
&
$\begin{array}{ll}   
e_1\dashv e_2=ae_1,\\
e_1\vdash e_2=be_1,\\
e_2\vdash e_1=ce_1,\\
e_2\vdash e_2=de_1
\end{array}$
&
$\begin{array}{ll}  
\alpha(e_2)=e_1,\\
\beta(e_2)=e_1
\end{array}$
\\ \hline  
$\mathcal{A}lg_4$
&
$\begin{array}{ll}   
e_1\dashv e_2=e_1,\\
e_2\dashv e_1=e_1,\\
e_2\dashv e_2=ae_1,\\
e_1\vdash e_2=be_1,\\
e_2\vdash e_1=ce_1,\\
e_2\vdash e_2=de_1,\\
\end{array}$
&
$\begin{array}{ll}  
\alpha(e_2)=e_1,\\
\beta(e_2)=e_1
\end{array}$
\\ \hline  
\end{tabular}	
\begin{remark}
 In two dimensional, all of the BiHom-associative dialgebras are Hom-associative dialgebras i.e. $\alpha=\beta$.
\end{remark}

\subsection{Three dimensional}

\begin{tabular}{||c||c||c||c||c||c||c||}
\hline
$Algebras$&Multiplications&Morphisms $\alpha,\beta$.
\\ \hline 
$\mathcal{A}lg_1$&
$\begin{array}{ll}  
e_1\dashv e_2=e_1,\\
e_2\dashv e_1=e_1,\\
e_2\dashv e_2=ae_1,\\
e_2\dashv e_3=be_1,\\ 
\end{array}$
$\begin{array}{ll}  
e_3\dashv e_2=ce_1,\\
e_2\vdash e_1=e_1,\\
e_2\vdash e_2=de_1,\\
e_3\vdash e_2=fe_1,\\
\end{array}$
&
$\begin{array}{ll}  
\alpha(e_2)=e_1\\
\beta(e_2)=e_1,\\
\beta(e_3)=be_3
\end{array}$
\\ \hline 
\end{tabular}

\begin{tabular}{||c||c||c||c||c||c||c||}
\hline
$Algebras$&Multiplications&Morphisms $\alpha,\beta$.
\\ \hline 
$\mathcal{A}lg_2$&
$\begin{array}{ll}  
e_1\dashv e_2=e_1,\\
e_2\dashv e_1=e_1,\\
e_2\dashv e_2=e_1,\\
e_2\dashv e_3=e_1,\\ 
e_3\dashv e_2=e_1,\\
\end{array}$
$\begin{array}{ll}  
e_1\vdash e_2=e_1,\\
e_2\vdash e_1=e_1,\\
e_2\vdash e_2=e_1,\\
e_3\vdash e_2=e_1,\\
\end{array}$
&
$\begin{array}{ll}  
\alpha(e_2)=e_1\\
\beta(e_2)=e_1,\\
\beta(e_3)=be_3
\end{array}$
\\ \hline 
$\mathcal{A}lg_3$&
$\begin{array}{ll}  
e_1\dashv e_2=e_1,\\
e_2\dashv e_1=e_,\\
e_2\dashv e_2=e_1,\\
e_2\dashv e_3=e_1,\\ 
e_3\dashv e_2=e_1,\\
\end{array}$
$\begin{array}{ll}  
e_1\vdash e_2=e_1,\\
e_2\vdash e_2=e_1,\\
e_2\vdash e_3=e_1,\\
e_3\vdash e_2=e_1,\\
\end{array}$
&
$\begin{array}{ll}  
\alpha(e_2)=e_1\\
\beta(e_2)=e_1,\\
\beta(e_3)=be_3
\end{array}$
% \\ \hline
% \end{tabular}
% 
% \begin{tabular}{||c||c||c||c||c||c||c||}
% \hline
% $Algebras$&Multiplications&Morphisms $\alpha,\beta$.
\\ \hline 
$\mathcal{A}lg_4$&
$\begin{array}{ll}  
e_1\dashv e_2=e_1,\\
e_2\dashv e_1=e_1,\\
e_2\dashv e_2=e_1,\\
e_2\dashv e_3=e_1,\\
\end{array}$
$\begin{array}{ll}  
e_1\vdash e_2=e_1,\\
e_2\vdash e_1=e_1,\\
e_2\vdash e_2=e_1,\\
e_2\vdash e_3=e_1,\\
e_3\vdash e_2=e_1,
\end{array}$
&
$\begin{array}{ll}  
\alpha(e_2)=e_1\\
\beta(e_2)=e_1,\\
\beta(e_3)=be_3
\end{array}$
% \\ \hline
% \end{tabular}
% 
% \begin{tabular}{||c||c||c||c||c||c||c||}
% \hline
% $Algebras$&Multiplications&Morphisms $\alpha,\beta$.
\\ \hline
$\mathcal{A}lg_5$&
$\begin{array}{ll}  
e_1\dashv e_2=e_1,\\
e_2\dashv e_1=e_1,\\
e_2\dashv e_2=e_1,\\
e_2\dashv e_3=e_1,\\
\end{array}$
$\begin{array}{ll}  
e_1\vdash e_2=e_1,\\
e_2\vdash e_1=e_1,\\
e_2\vdash e_3=e_1,\\
e_3\vdash e_2=e_1,
\end{array}$
&
$\begin{array}{ll}  
\alpha(e_2)=e_1\\
\beta(e_2)=e_1,\\
\beta(e_3)=be_3
\end{array}$
\\ \hline
\end{tabular}

\subsection{Four dimensional}

\begin{tabular}{||c||c||c||c||c||c||c||}
\hline
$Algebras$&Multiplications&Morphisms $\alpha,\beta$.
\\ \hline 
$\mathcal{A}lg_1$&
$\begin{array}{ll}  
e_2\dashv e_1=e_4,\\
e_2\dashv e_3=e_4,\\
e_3\dashv e_1=e_4,\\
e_3\dashv e_2=e_4,\\
e_4\dashv e_4=e_4,\\
\end{array}$
$\begin{array}{ll}  
e_1\vdash e_2=e_4,\\
e_2\vdash e_2=ce_4,\\
e_3\vdash e_3=e_4,\\
e_3\vdash e_4=de_3,\\
\end{array}$
&
$\begin{array}{ll}  
\alpha(e_2)=be_2\\
\beta(e_2)=e_1,\\
\end{array}$
$\begin{array}{ll}  
\beta(e_3)=e_2,\\
\beta(e_4)=e_3,\\
\end{array}$
\\ \hline 
$\mathcal{A}lg_2$&
$\begin{array}{ll}  
e_1\dashv e_2=e_4,\\
e_1\dashv e_4=e_4,\\
e_2\dashv e_1=ae_4,\\
e_2\dashv e_3=be_4,\\
e_3\dashv e_1=-ce_4,\\
e_3\dashv e_2=e_4,\\
\end{array}$
$\begin{array}{ll}  
e_1\vdash e_2=e_4,\\
e_2\vdash e_2=de_4,\\
e_3\vdash e_3=fe_4,\\
e_3\vdash e_4=e_4,\\
e_4\vdash e_4=e_4,\\
\end{array}$
&
$\begin{array}{ll}  
\alpha(e_2)=e_2\\
\beta(e_2)=e_1,\\
\end{array}$
$\begin{array}{ll}  
\beta(e_3)=e_2,\\
\beta(e_4)=e_3,\\
\end{array}$
\\ \hline 
$\mathcal{A}lg_3$&
$\begin{array}{ll}  
e_1\dashv e_4=e_4,\\
e_2\dashv e_1=e_4,\\
e_2\dashv e_2=e_4,\\
e_2\dashv e_3=be_4,\\
e_3\dashv e_1=ce_4,\\
e_3\dashv e_2=e_4,\\
\end{array}$
$\begin{array}{ll}  
e_1\vdash e_2=e_4,\\
e_2\vdash e_2=e_4,\\
e_3\vdash e_2=ce_4,\\
e_3\vdash e_3=de_4,\\
e_4\vdash e_4=e_4,\\
\end{array}$
&
$\begin{array}{ll}  
\alpha(e_2)=e_2\\
\alpha(e_3)=e_3\\
\end{array}$
$\begin{array}{ll}  
\beta(e_2)=e_1,\\
\beta(e_3)=e_2,\\
\beta(e_4)=e_3,\\
\end{array}$
% \\ \hline
% \end{tabular}
% 
% \begin{tabular}{||c||c||c||c||c||c||c||}
% \hline
% $Algebras$&Multiplications&Morphisms $\alpha,\beta$.
\\ \hline 
$\mathcal{A}lg_4$&
$\begin{array}{ll}  
e_1\dashv e_4=e_4,\\
e_2\dashv e_2=ae_4,\\
e_2\dashv e_3=e_4,\\
e_3\dashv e_1=e_4,\\
e_3\dashv e_2=ce_4,\\
\end{array}$
$\begin{array}{ll}  
e_3\dashv e_3=e_4,\\
e_1\vdash e_2=e_4,\\
e_2\vdash e_2=e_4,\\
e_3\vdash e_3=e_4,\\
e_4\vdash e_4=e_4,\\
\end{array}$
&
$\begin{array}{ll}  
\alpha(e_2)=e_2\\
\alpha(e_3)=e_3\\
\end{array}$
$\begin{array}{ll}  
\beta(e_2)=e_1,\\
\beta(e_3)=e_2,\\
\beta(e_4)=e_3,\\
\end{array}$
\\ \hline
\end{tabular}

\begin{tabular}{||c||c||c||c||c||c||c||}
\hline
$Algebras$&Multiplications&Morphisms $\alpha,\beta$.
\\ \hline 
$\mathcal{A}lg_5$&
$\begin{array}{ll}  
e_1\dashv e_4=e_4,\\
e_2\dashv e_2=e_4,\\
e_2\dashv e_3=e_4,\\
e_3\dashv e_4=e_4,\\
e_3\dashv e_2=e_4,\\

\end{array}$
$\begin{array}{ll}  
e_3\dashv e_3=e_4,\\
e_3\dashv e_4=e_4,\\
e_1\vdash e_3=e_4,\\
e_2\vdash e_2=e_4,\\
e_3\vdash e_3=e_4,\\
\end{array}$
&
$\begin{array}{ll}  
\alpha(e_2)=e_2\\
\alpha(e_4)=e_4\\
\end{array}$
$\begin{array}{ll}  
\beta(e_2)=e_1,\\
\beta(e_3)=e_2,\\
\beta(e_4)=e_3,\\
\end{array}$
\\ \hline 
$\mathcal{A}lg_6$&
$\begin{array}{ll}  
e_2\dashv e_2=e_4,\\
e_2\dashv e_3=e_4,\\
e_3\dashv e_2=e_4,\\
e_3\dashv e_3=e_4,\\
e_3\dashv e_4=e_4,\\
\end{array}$
$\begin{array}{ll}  
e_1\vdash e_3=e_4,\\
e_1\vdash e_4=e_4,\\
e_2\vdash e_2=e_4,\\
e_3\vdash e_1=e_4,\\
e_3\vdash e_3=e_4,\\
\end{array}$
&
$\begin{array}{ll}  
\alpha(e_2)=e_2\\
\alpha(e_4)=e_4
\end{array}$
$\begin{array}{ll}  
\beta(e_2)=e_1,\\
\beta(e_3)=e_2,\\
\beta(e_4)=e_3,\\
\end{array}$
\\ \hline 
$\mathcal{A}lg_7$&
$\begin{array}{ll}  
e_1\dashv e_2=e_4,\\
e_1\dashv e_4=e_4,\\
e_2\dashv e_2=e_4,\\
e_2\dashv e_4=fe_4,\\
e_3\dashv e_3=-ge_4,\\
\end{array}$
$\begin{array}{ll}  
e_1\vdash e_4=e_4,\\
e_2\vdash e_2=e_4,\\
e_2\vdash e_3=e_4,\\
e_3\vdash e_1=e_4,\\
e_3\vdash e_2=-he_4,\\
e_3\vdash e_3=ke_4,\\
\end{array}$
&
$\begin{array}{ll}  
\alpha(e_3)=e_3,\\
\alpha(e_4)=e_4\\
\end{array}$
$\begin{array}{ll}  
\beta(e_2)=e_1,\\
\beta(e_3)=e_2,\\
\beta(e_4)=e_3,\\
\end{array}$
% \\ \hline
% \end{tabular}
% 
% 
% \begin{tabular}{||c|c|c|c|c|c|c||}
% \hline
% $Algebras$&Multiplications&Morphisms $\alpha,\beta$.
\\ \hline  
$\mathcal{A}lg_8$&
$\begin{array}{ll}  
e_1\dashv e_3=e_4,\\
e_1\dashv e_4=e_4,\\
e_2\dashv e_2=e_4,\\
e_2\dashv e_4=e_4,\\
e_3\dashv e_3=e_4,\\
\end{array}$
$\begin{array}{ll}  
e_1\vdash e_3=e_4,\\
e_3\vdash e_1=e_4,\\
e_3\vdash e_2=e_4,\\
e_3\vdash e_3=e_4,\\
\end{array}$
&
$\begin{array}{ll}  
\alpha(e_3)=e_3\\
\alpha(e_4)=e_4\\
\end{array}$
$\begin{array}{ll}  
\beta(e_2)=e_1,\\
\beta(e_3)=e_2,\\
\beta(e_4)=e_3,\\
\end{array}$
\\ \hline 
$\mathcal{A}lg_9$&
$\begin{array}{ll}  
e_2\dashv e_2=e_1+e_4,\\
e_2\dashv e_3=e_1+e_4,\\
e_3\dashv e_2=e_1+e_4,\\
e_4\dashv e_2=e_1+e_4,\\
\end{array}$
$\begin{array}{ll}  
e_1\vdash e_2=-e_1+e_4,\\
e_2\vdash e_2=e_1,\\
e_3\vdash e_3=e_1+e_4,\\
e_4\vdash e_2=e_1+e_4,\\
\end{array}$
&
$\begin{array}{ll}  
\alpha(e_2)=e_1\\
\alpha(e_3)=e_2\\
\alpha(e_4)=e_4\\
\end{array}$
$\begin{array}{ll}  
\beta(e_3)=e_3,\\
\beta(e_4)=e_4,\\
\end{array}$
\\ \hline
$\mathcal{A}lg_{10}$&
$\begin{array}{ll}  
e_1\dashv e_2=e_4,\\
e_2\dashv e_2=e_1+e_4,\\
e_2\dashv e_3=e_4,\\
e_3\dashv e_2=e_1,\\
e_3\dashv e_3=e_4,\\
\end{array}$
$\begin{array}{ll}  
e_4\dashv e_2=e_4,\\
e_1\vdash e_2=e_4,\\
e_2\vdash e_2=e_1,\\
e_3\vdash e_3=e_1+e_4,\\
e_4\vdash e_2=e_1+e_4,\\
\end{array}$
&
$\begin{array}{ll}  
\alpha(e_2)=e_1\\
\alpha(e_3)=e_2\\
\alpha(e_4)=e_4\\
\end{array}$
$\begin{array}{ll}  
\beta(e_3)=e_3,\\
\beta(e_4)=e_4,\\
\end{array}$
\\ \hline 
$\mathcal{A}lg_{11}$&
$\begin{array}{ll}  
e_2\dashv e_2=fe_1+ge_4,\\
e_2\dashv e_3=e_4,\\
e_3\dashv e_2=e_1+e_4,\\
e_3\dashv e_3=e_4,\\
e_4\dashv e_2=e_4,\\
\end{array}$
$\begin{array}{ll}  
e_1\vdash e_2=e_4,\\
e_2\vdash e_2=he_1-ke_4,\\
e_3\vdash e_3=e_1+e_4,\\
e_4\vdash e_2=e_1+e_4,\\
\end{array}$
&
$\begin{array}{ll}  
\alpha(e_2)=e_1\\
\alpha(e_3)=e_2\\
\alpha(e_4)=e_4\\
\end{array}$
$\begin{array}{ll}  
\beta(e_3)=e_3,\\
\beta(e_4)=e_4,\\
\end{array}$
% \\ \hline
% \end{tabular}
% 
% \begin{tabular}{||c||c||c||c||c||c||c||}
% \hline
% $Algebras$&Multiplications&Morphisms $\alpha,\beta$.
\\ \hline 
$\mathcal{A}lg_{12}$&
$\begin{array}{ll}  
e_1\dashv e_4=e_4,\\
e_2\dashv e_2=e_4,\\
e_2\dashv e_3=ae_4,\\
e_2\dashv e_4=e_4,\\
\end{array}$
$\begin{array}{ll}  
e_3\dashv e_3=e_4,\\
e_1\vdash e_2=e_4,\\
e_2\vdash e_2=e_4,\\
e_2\vdash e_3=-be_4,\\
e_3\vdash e_2=e_4,\\
\end{array}$
&
$\begin{array}{ll}  
\alpha(e_3)=e_3\\
\alpha(e_4)=e_4\\
\beta(e_1)=e_1,\\
\end{array}$
$\begin{array}{ll}  
\beta(e_2)=e_1+e_2,\\
\beta(e_3)=e_2+e_3,\\
\beta(e_4)=e_3+e_4,\\
\end{array}$
\\ \hline
$\mathcal{A}lg_{13}$&
$\begin{array}{ll}  
e_1\dashv e_2=e_4,\\
e_1\dashv e_3=e_4,\\
e_2\dashv e_1=e_4,\\
e_2\dashv e_2=e_4,\\
e_2\dashv e_3=e_4,\\
\end{array}$
$\begin{array}{ll}  
e_3\dashv e_1=e_4,\\
e_1\vdash e_2=e_4,\\
e_2\vdash e_2=e_4,\\
e_2\vdash e_3=e_4,\\
e_3\vdash e_3=e_4,\\
\end{array}$
&
$\begin{array}{ll}  
\alpha(e_2)=e_2\\
\alpha(e_3)=e_3\\
\beta(e_1)=e_1,\\
\end{array}$
$\begin{array}{ll}  
\beta(e_2)=e_1+e_2,\\
\beta(e_3)=e_2+e_3,\\
\beta(e_4)=e_3+e_4,\\
\end{array}$
\\ \hline 
$\mathcal{A}lg_{14}$&
$\begin{array}{ll}  
e_1\dashv e_1=e_4,\\
e_1\dashv e_3=-ce_4,\\
e_2\dashv e_2=e_4,\\
e_2\dashv e_3=e_4,\\
e_3\dashv e_1=e_4,\\
e_3\dashv e_2=e_4,\\
\end{array}$
$\begin{array}{ll}
e_3\dashv e_3=-2ae_4,\\  
e_1\vdash e_2=e_4,\\
e_2\vdash e_2=e_4,\\
e_2\vdash e_3=e_4,\\
e_3\vdash e_2=e_4,\\
e_3\vdash e_3=be_4,\\
\end{array}$
&
$\begin{array}{ll}  
\alpha(e_1)=e_1\\
\alpha(e_2)=e_2\\
\beta(e_1)=e_1,\\
\end{array}$
$\begin{array}{ll}  
\beta(e_2)=e_1+e_2,\\
\beta(e_3)=e_2+e_3,\\
\beta(e_4)=e_3+e_4,\\
\end{array}$
\\ \hline 
\end{tabular}

\begin{tabular}{||c||c||c||c||c||c||c||}
\hline
$Algebras$&Multiplications&Morphisms $\alpha,\beta$.
\\ \hline 
$\mathcal{A}lg_{15}$&
$\begin{array}{ll}  
e_1\dashv e_1=-e_4,\\
e_1\dashv e_2=ae_4,\\
e_2\dashv e_3=be_4,\\
e_3\dashv e_1=ce_4,\\
e_3\dashv e_2=de_4,\\
e_3\dashv e_3=e_4,\\
\end{array}$
$\begin{array}{ll}  
e_1\vdash e_2=fe_4,\\
e_1\vdash e_4=e_4,\\
e_2\vdash e_2=e_4,\\
e_2\vdash e_3=e_4,\\
e_3\vdash e_2=ge_4,\\
e_3\vdash e_3=e_4,\\
e_3\vdash e_4=e_4,\\
\end{array}$
&
$\begin{array}{ll}  
\alpha(e_2)=e_2\\
\beta(e_1)=e_1,\\
\end{array}$
$\begin{array}{ll}  
\beta(e_2)=e_1+e_2,\\
\beta(e_3)=e_2+e_3,\\
\beta(e_4)=e_3+e_4,\\
\end{array}$
\\ \hline 
$\mathcal{A}lg_{16}$&
$\begin{array}{ll}  
e_1\dashv e_2=e_4,\\
e_2\dashv e_1=e_4,\\
e_2\dashv e_2=e_4,\\
e_2\dashv e_3=ae_4,\\
e_2\dashv e_4=e_4,\\
e_3\dashv e_2=e_4,\\
\end{array}$
$\begin{array}{ll}  
e_1\vdash e_2=be_4,\\
e_2\vdash e_2=ce_4,\\
e_3\vdash e_2=de_4,\\
e_3\vdash e_3=e_4,\\
e_3\vdash e_4=e_4,\\
\end{array}$
&
$\begin{array}{ll}  
\alpha(e_1)=ae_1\\
\beta(e_1)=e_1,\\
\end{array}$
$\begin{array}{ll}  
\beta(e_2)=e_1+e_2,\\
\beta(e_3)=e_2+e_3,\\
\beta(e_4)=e_3+e_4,\\
\end{array}$
\\ \hline 
\end{tabular}

   \section{Derivation of BiHom-associative dialgebras}
In this section, we introduce and study derivations of BiHom-dendrifom, BiHom-dialgebras.
\begin{definition}
Let $(A,\mu, \alpha, \beta)$ be a BiHom-associative algebra. A linear map ${D} : A\longrightarrow  A$ 
is called an $(\alpha^s, \beta^r)$-derivation of 
$(A,\mu, \alpha, \beta)$, if it satisfies  
$$
{D}\circ\alpha=\alpha\circ {D}\quad \text{and}\quad{D}\circ\beta=\beta\circ {D} 
$$
$$
{D}\circ\mu(x, y)=\mu({D}(x), \alpha^s \beta^r(y))+\mu(\alpha^s \beta^r(x), {D}(y))
$$
\end{definition}
\begin{example}
We consider the $2$-dimensional BiHom-associative with a basis $\left\{e_1, e_2\right\}$.
For $\mu(e_1,e_1)=-e_1,\quad \mu(e_1, e_2)=-e_2,\quad \mu(e_2, e_1)=0,\quad \mu(e_2, e_2)=e_2$ and \\
$\alpha(e_1)=e_1,\quad\alpha(e_2)=-e_2,\quad \beta(e_1)=e_1,\quad \beta(e_2)=e_2.$ A direct computation gives that : 
${D}(e_1)=d_{22}e_1,\quad {D}(e_2)=d_{22}e_2,$\\
$\alpha^s(e_1)=\frac{\alpha_{21}\beta_{22}}{\beta_{21}}e_1+\frac{e_2}{2\beta_{21}},\quad\alpha^s(e_2)=\alpha_{21}e_1,\quad
\beta^r(e_1)=\frac{e_2}{2\beta_{21}}e_2,\quad\beta^r(e_2)=\beta_{21}e_1+\beta_{22}e_2$.
\end{example}

\begin{definition}
Let $({D}, \dashv, \vdash, \alpha, \beta)$ be a BiHom-associative dialgebra. A linear map ${D} : {D}\rightarrow {D}$ is called an 
$(\alpha^k, \beta^l)$-derivation of ${D}$ if it satisfies 
\begin{enumerate}
	\item [$1.$] ${D}\circ\alpha=\alpha\circ{D},\,{D}\circ\beta=\beta\circ{D}$;
	\item [$2.$] ${D}(x\dashv y)=\alpha^k\beta^l(x)\dashv{D}(y)+{D}(x)\dashv\alpha^k\beta^l(y);$
	\item [$3.$] ${D}(x\vdash y)=\alpha^k\beta^l(x)\vdash {D}(y)+{D}(x)\vdash\alpha^k\beta^l(y),$
\end{enumerate}
for $x, y\in   {D}.$

We denote by $Der({D}):=\displaystyle\bigoplus_{k\geq 0}\displaystyle\bigoplus_{l\geq 0}Der_{(\alpha^k, \beta^l)}({D})$, where
 $Der_{(\alpha^k, \beta^l)}({D})$ is the set of all $(\alpha^k,\beta^l)$-derivations of ${D}$.
\end{definition} 
\begin{proposition}
For any ${D}\in Der_{(\alpha^s, \beta^r)}(A)$ and ${D}'\in Der_{(\alpha^{s'},\beta^{r'})}(A)$, we have 
$\left[{D},{D'}\right]\in Der_{(\alpha^{s+s'}, \beta^{r+r'})}(A)$.
\end{proposition}
\begin{proof}
For $x, y\in A$, we have 
$$\begin{array}{ll}
\left[{D}, {D'}\right]\circ\mu(x, y)
&={D}\circ{D'}\circ\mu(x, y)-{D'}\circ{D}\circ\mu(x, y)\\
&={D}(\mu({D}'(x), \alpha^s\beta^r(y))+\mu(\alpha^s\beta^r(x),{D}'(y)))\\
&-{D}'(\mu({D}(x), \alpha^s\beta^r(y))+\mu(\alpha^s\beta^r(x),{D}(y)))\\
&=\mu({D}\circ{D'}(x),\alpha^{s+s'}\beta^{r+r'}(y))+\mu(\alpha^s\beta^r\circ{D'}(x),{D}\circ\alpha^s\beta^r(y))\\
&+\mu({D}\circ\alpha^s\beta^r(x),\alpha^s\beta^r\circ{D}'(y))+\mu(\alpha^{s+s'}\beta^{r+r'}(x),{D}\circ{D'}(y))\\
&-\mu({D'}\circ{D}(x),\alpha^{s+s'}\beta^{r+r'}(y))-\mu(\alpha^s\beta^r\circ{D}(x),{D'}\circ\alpha^s\beta^r(y))\\
&-\mu({D'}\circ\alpha^s\beta^r(x),\alpha^s\beta^r{D}(y))-\mu(\alpha^{s+s'}\beta^{r+r'}(x),{D'}\circ{D}(y)).
\end{array}$$
Since ${D}$ and ${D}'$ satisfy 
${D}\circ\alpha=\alpha\circ{D},\,{D}'\circ\alpha=\alpha\circ{D}'$,\,  
${D}\circ\beta=\beta\circ{D},\,{D}'\circ\beta=\beta\circ{D}'$.\\ We obtain  
$\alpha^s\beta^r\circ{D'}={D'}\circ\alpha^s\beta^r,\, {D}\circ\alpha^{s'}\beta^{r'}=\alpha^{s'}\beta^{r'}\circ{D}.$
Therefore, we arrive at\\
$\left[{D},{D'}\right]\circ\mu(x, y)=\mu(\alpha^{s+s'}\beta^{r+r'}(x),\left[{D},{D'}\right](y))+
\mu(\left[{D},{D'}\right](x),\alpha^{s+s'}\beta^{r+r'}(y)).$

Furthermore, it is straightforward to see that 
$$\begin{array}{ll}
\left[{D},{D'}\right]\circ\alpha
&={D}\circ{D'}\circ\alpha-{D'}\circ{D}\circ\alpha\\
&=\alpha\circ{D}\circ{D}'-\alpha\circ{D}'\circ{D}=\alpha\circ\left[{D},{D'}\right].
\end{array}$$ 
$$\begin{array}{ll}
\left[{D},{D'}\right]\circ\beta
&={D}\circ{D'}\circ\beta-{D'}\circ{D}\circ\beta\\
&=\beta\circ{D}\circ{D}'-\beta\circ{D}'\circ{D}=\beta\circ\left[{D},{D'}\right]
\end{array}$$
which yields that $\left[{D},{D'}\right]\in Der_{(\alpha^{s+s'}, \beta^{r+r'})}(A)$ with $\mu=\dashv=\vdash.$
\end{proof}

\begin{proposition}
The space $Der_{(\alpha^{s}, \beta^{r})}(A)$ is an invariant of the triple BiHom-associative algebra A.
\end{proposition}
\begin{proof}
Let $\sigma :  (A, \dashv_A, \vdash_A,  \alpha^s, \beta^r)\longrightarrow (B, \dashv_B, \vdash_B, \alpha^s, \beta^r)$  be a triple BiHom-associative algebra isomorphism
and let ${D}$ be a $(\alpha^s, \beta^r)$-derivation of A. Then for any $x, y, z\in B$. We have :  
$$\begin{array}{ll}
\sigma{D}\sigma^{-1}\circ(((x)\dashv_B(y))\dashv_B (z))
&=\sigma{D}\circ((\sigma^{-1}(x)\dashv_A\sigma^{-1}(y))\dashv_A\sigma^{-1}(z))\\
&=\sigma({D}\circ\sigma^{-1}(x)\vdash_A\sigma^{-1}\circ\alpha^s\beta^r(y))\vdash_A\sigma^{-1}\circ\alpha^s\beta^r(z))\\
&+\sigma(\sigma^{-1}\circ\alpha^s\beta^r(x)\vdash_A{D}\circ\sigma^{-1}(y))\vdash_A\sigma^{-1}\circ\alpha^s\beta^r(z))\\
&+\sigma(\sigma^{-1}\circ\alpha^s\beta^r(x)\vdash_A\sigma^{-1}\circ\alpha^s\beta^r(y)\vdash_A{D}\circ\sigma^{-1}(z))\\
&=({D}\circ\sigma^{-1}(x)\dashv_B\alpha^s\beta^r(y))\dashv_B\alpha^s\beta^r(z))\\
&+(\alpha^s\beta^r(x)\dashv_B\sigma\circ{D}\circ\sigma^{-1}(y))\dashv_B\alpha^s\beta^r(z))\\
&+(\alpha^s\beta^r(x)\dashv_B\alpha^s\beta^r(y))\dashv_B{D}\circ\sigma^{-1}(z)).
\end{array}$$
Thus $\sigma\circ{D}\circ\sigma^{-1}$ is a $(\alpha^s, \beta^r)$-derivation of $B$, hence the mapping.
$\psi : Der_{(\alpha^{s},\beta^{r})}(A)\longrightarrow Der_{(\alpha^{s}, \beta^{r})}(B)$,\\ ${D}\longmapsto \sigma{D}\sigma^{-1}$
is an isomorphism of triple  BiHom-associative algebras. 

In fact, it is easy to see that $\psi$ is linear. Moreover let 
${D}_1, {D}_2, {D}_3$ be derivations of A : 
$$\begin{array}{ll}
&\alpha^s\beta^r\circ\psi({D}_1\dashv_{tr}{D}_2)\dashv_{tr}{D}_3)=\\
&=\alpha^s\beta^r\psi(tr({D}_1)({D}_2\dashv{D}_3))+\alpha^s\beta^r\psi(tr({D}_3)({D}_1{D}_2))
+\alpha^s\beta^r\psi(tr({D}_2)({D}_3\dashv {D}_1))\\
&=\alpha^s\beta^r tr({D}_1)\psi({D}_2\dashv{D}_3)
+\alpha^s\beta^r tr({D}_3)\psi({D}_1\dashv{D}_2)+\alpha^s\beta^r  tr({D}_2)\psi({D}_3\dashv{D}_1)\\
&=\alpha^s\beta^r tr(\psi({D}_1))(\psi({D}_2)\dashv\psi({D}_3))+\alpha^s\beta^r tr(\psi({D}_3))(\psi({D}_1)\dashv\psi({D}_2))\\
&+\alpha^s\beta^r tr(\psi({D}_2))\psi((\psi({D}_3)\dashv\psi({D}_1)),
\end{array}$$
since $\psi$ is a morphism of the $Der_{(\alpha^{s},\beta^{r})}(A)$ and $Der_{(\alpha^{s},\beta^{r})}(B)$, and 
$tr({D})=tr(\sigma\circ{D}\circ\sigma^{-1}).$\\
Then $\alpha^s\beta^r\psi(({D}_1\dashv_{tr}{D}_2)\dashv_{tr}{D}_3))=\alpha^s\beta^r((\psi({D}_1)\dashv_{tr}\psi({D}_2))\dashv_{tr}\psi({D}_3)).$
\end{proof}

% {\bf Acknowlegments :} 

\end{document}